\documentclass[a4paper,11pt]{amsart}

\usepackage{hyperref}

\usepackage{amsmath,amsthm,amssymb,latexsym,amsfonts,wrapfig,yfonts,mathrsfs,mathtools}
\usepackage[latin1]{inputenc}
\usepackage[english,activeacute]{babel}
\usepackage{graphicx,xcolor,esint}
\usepackage[mathscr]{euscript}
\usepackage{geometry,upgreek,bbm} 
\usepackage[normalem]{ulem}
\usepackage{subfigure}
\usepackage{enumerate}
 \usepackage{mathrsfs}

\usepackage{parskip}
\usepackage{hyperref}
\usepackage{cleveref}

\usepackage{todonotes}

\newcommand{\norm}[1]{\left\lVert#1\right\rVert}
\allowdisplaybreaks

\def \N{\mathbb N}
\def \Z{\mathbb Z}

\def \R{\mathbb R}
\def \C{\mathbb C}

\def \pa{{\partial}}

\def \O{\mathcal{O}}

\numberwithin{equation}{section}

\theoremstyle{plain}
\newtheorem{thm}{Theorem}[section]
\newtheorem{lem}[thm]{Lemma}
\newtheorem{prop}[thm]{Proposition}

\newtheorem{rk}[thm]{Remark}
\theoremstyle{definition}

\title[Continuum limit for DNLS with memory effect]{Continuum limit for discrete NLS with memory effect}

\author[R. Grande]{Ricardo Grande}
\address{International School for Advanced Studies (SISSA), Via Bonomea 265, 34136, Trieste, Italy}
\email{rgrandei@sissa.it} 

\begin{document}

\begin{abstract}
We consider a discrete nonlinear Schr\"odinger equation with long-range interactions and a memory effect on the infinite lattice $h\Z$ with mesh-size $h>0$. Such models are common in the study of charge and energy transport in biomolecules. Given that the distance between base pairs is small, we consider the continuum limit: a sharp approximation to the system as $h\rightarrow 0$. In this limit, we prove that solutions to this discrete equation converge strongly in $L^2$ to the solution to a continuous NLS-type equation with a memory effect, and we compute the precise rate of convergence. In order to obtain these results, we generalize some recent ideas proposed by Hong and Yang in $L^2$-based spaces to classical functional settings in dispersive PDEs involving the smoothing effect and maximal function estimates, as originally introduced in the pioneering works of Kenig, Ponce and Vega. We believe that our approach may therefore be adapted to tackle continuum limits of more general dispersive equations.
\end{abstract}

\maketitle

\thispagestyle{empty}

\section{Introduction}

\subsection{Background}

The study of nonlinear dispersive equations on lattices is of fundamental importance to model charge and energy transport in biomolecules \cite{bioDNA}. From a computational and analytical perspective, the analysis of such equations is quite challenging given their size and mathematical complexity. One way to simplify this analysis is by exploiting two aspects of such models: the large amount of base pairs in these biomolecules and the small distance between them. In particular, one typically considers an infinite lattice of equispaced base pairs with mesh size $h>0$ and considers the limit as $h\rightarrow 0$. This continuum limit can often be shown to solve a simpler (continuous) PDE which captures the main aspects of the discrete system. 

One of the most common models in biological physics is the following:
\begin{equation}\label{eq:continuumKLS}
 \left\lbrace \begin{array}{ll}
i \pa_t u_h (t,x_m) = h \sum_{n\neq m} J_{n-m}\, [u_h (t,x_n) - u_h (t,x_m)] \pm |u_h (t,x_m)|^2 u_h (t,x_m),\\
u_h |_{t=0} =f_h.
\end{array} \right.
\end{equation}
where $u_h:[0,T]\times h\Z\rightarrow \C$ is the wave function, and where $x_m=hm\in h\Z$ live in the one-dimensional lattice with mesh size $h>0$. The cubic nonlinearity represents a four-wave interaction, where a $+$ sign corresponds to a repulsive on-site self-interaction, and $-$ corresponds to the focusing case.  One often assumes that the initial distribution $f_h$ is the discretization of some continuous function $f:\R\rightarrow \C$, which is defined as follows:
\begin{equation}\label{eq:intro_discretization}
 f_h (x_m)=\frac{1}{h}\int_{x_m}^{x_{m+1}} f (x) \, dx \qquad \mbox{for}\ m\in\Z.
\end{equation}

The interactions between different base pairs are modelled by the kernel $\{ J_n \}_{n\in\Z}$. A typical modeling choice \cite{bioDNA} is a law that is inversely proportional to some power of the distance between them:
\begin{equation}\label{eq:intro_J}
 J_{m-n}:= |x_m-x_n|^{-1-\alpha}\quad \mbox{for}\ m\neq n\in\Z.
 \end{equation}

The model \eqref{eq:continuumKLS} was first studied by Kirkpatrick, Lenzmann and Staffilani in the context of quantum mechanics \cite{KLS}. The main goal is to derive a \emph{continuum limit}, i.e. to obtain a simpler continuous model  that approximates the dynamics of the discrete model \eqref{eq:continuumKLS}  as the distance between base pairs $h$ tends to zero. Such continuum limits have since then been proved in the context of the NLS equation with a discrete Laplacian \cite{Hong,Quentin}, the NLS equation on a large but finite lattice \cite{HKY}, fractional NLS in 2D \cite{ChoiAceves}, dispersion-managed nonlinear NLS \cite{Choi}, the Ablowitz-Ladik system \cite{KillipVisan}, and the Klein-Gordon equation \cite{Quentin2}, among others. 

In \cite{KLS}, Kirkpatrick, Lenzmann and Staffilani show that under mild technical conditions, the asymptotic behavior of the interactions \eqref{eq:intro_J} is all that matters. In other words, if
\begin{equation}\label{eq:Jdecay}
 \lim_{n\rightarrow \infty} |x_n|^{1+\alpha} J_n = C_{\alpha}>0,
 \end{equation}
for $\alpha\in (1,2)$, then the continuum limit of the solution to \eqref{eq:continuumKLS} (as $h\rightarrow 0$) is the solution to the fractional cubic NLS equation:
\begin{equation}\label{eq:intro_NLS}
 \left\lbrace \begin{array}{ll}
 i\,\pa_t u = c_{\alpha}\, (-\Delta)^{\frac{\alpha}{2}} u \pm |u|^2 u, \quad x\in \R, \\
u |_{t=0} = f.
\end{array} \right.
\end{equation}
Such choice of $\alpha\in (1,2)$ is a common in order to model long-range interactions between base pairs, see \cite{bioDNA}. It is interesting to mention that when $\alpha\geq 2$, the interactions decay so fast that only local effects survive in the continuum limit. In this case, the continuum limit is the solution to the cubic NLS equation with the usual Laplacian (but note that condition \ref{eq:Jdecay} must be slightly modified).

Unfortunately, there are some limitations to these results. From a mathematical viewpoint, the regularity of the initial discrete distribution $f_h$ in \eqref{eq:continuumKLS} is rather high compared to the continuous theory for NLS equations. This is due to the use of the Sobolev embedding in the proof of local well-posedness of the discrete equation. From a computational viewpoint, the convergence of the discrete solution $u_h$ to $u$ in $L^{\infty}_t ([0,T],H^{\alpha/2}_x (\R))$ is weak, due to the use of the Banach-Alaoglu theorem. It would be desirable to obtain strong convergence and precise bounds on the rate of convergence for the purpose of implementation of algorithms as well as other practical applications.

The first issue is a consequence of the bad behavior of dispersive properties on the lattice $h\Z$, as originally observed by Ignat and Zuazua \cite{zuazua}. In particular, one cannot blindly apply much of the usual machinery developped for dispersive equations in the 80s and 90s, such as Strichartz estimates and local smoothing estimates. 

The second issue mentioned before, concerning weak convergence to the limit, was addressed in a recent work of Hong and Yang \cite{Hong}. In this paper, the authors prove that the linear interpolation of the solution to the discrete problem $u_h$ converges to the solution to the continuous problem $u$ in $L^2_x$. Their techniques, which they recently extended to the torus \cite{HongTorus}, rely heavily on the use of $L^2_x$-based spaces.

One of the goals of this paper is to generalize this approach to more general functional spaces which allow us to exploit additional dispersive properties such as the smoothing effect and maximal function estimates, introduced in the 90's the pioneering works of Kenig, Ponce and Vega \cite{KPV}. In order to test our techniques, we consider a discrete toy model whose local well-posedness theory requires the full use of this classical functional setting. After establishing the well-posedness of the discrete system, we prove strong convergence to a continuum limit by combining estimates from the well-posedness theory and a bootstrap argument, leading to explicit rates of convergence. We believe that these techniques can be applied to study the convergence of a broad range of discrete dispersive problems. 

Let us also mention that our system \eqref{eq:continuumKLS} is, for most choices of $J_n$, not integrable. This prevents us from using techniques based on conservation laws for integrable lattice models, such as those in \cite{HGKillipVisan} and \cite{KillipVisan}. 

\subsection{Statement of results}

We consider the following model to test our techniques:
\begin{equation}\label{eq:intro_discrete}
 \left\lbrace \begin{array}{ll}
 i^{\beta}\, \pa_t^{\beta} u_h = (-\Delta_h)^{\frac{\alpha}{2}} u_h \pm \Pi_h\, R_h \left( |u_h|^{p-1} u_h\right),\quad (t,x)\in [0,T]\times h\Z,\\
 u_h |_{t=0}=\Pi_h f_{2h},
 \end{array} \right.
\end{equation}
where the operator $(-\Delta_h)^{\frac{\alpha}{2}}$ is defined in agreement with \eqref{eq:intro_J} in order to model long-range interactions:
 \[
 (-\Delta_h)^{\frac{\alpha}{2}} u_h (x_m) := h \, \sum_{n\neq m} \frac{u(t,x_n) - u(t,x_m)}{|x_n - x_m|^{1+\alpha}}.
 \]
The operator $\pa_t^{\beta}$, $\beta\in (0,1)$, denotes the Caputo derivative:
\[ \pa_t^{\beta} u (t)= \frac{1}{\Gamma (1-\beta)} \int_0^t \frac{\pa_{\tau} u(\tau)}{(t-\tau)^{\beta}}\, d\tau.\]
A fractional $\alpha$ represents long-range interactions, while a fractional $\beta$ accounts for a memory effect. The equation corresponding to $\beta=1$, was first proposed by Laskin in \cite{Laskin} as the fundamental equation in fractional quantum mechanics, and exactly corresponds to the model \eqref{eq:continuumKLS}. This equation is derived by considering a L\'evy distribution on the set of all possible paths for a quantum particle, as opposed to the Gaussian distribution present in the Feynmann path-integral. The case of fractional $\beta$ was first proposed by Naber in \cite{Naber} by allowing the evolution to be non-Markovian\footnote{Because of its connection with stochastic processes \cite{prob}, the Caputo derivative $\pa_t^{\beta}$ has been used to model various phenomena involving memory effects in physics \cite{Negrete1}, and economics \cite{econ}. A similar memory effect has also been studied in connection to the porous medium equation, as exemplified by the work of Allen, Caffarelli and Vasseur \cite{fracPME1,fracPME2}.}, thus giving rise to the memory effect.

The ``well-prepared'' initial datum $f_{2h}$ is the discretization of some continuous function $f$ as defined in \eqref{eq:intro_discretization},
while the operator $\Pi_h$ is the discrete interpolator:
\begin{equation}\label{eq:intro_interpolation}
\Pi_h f_{2h}(x_{2m+1})= \frac{f_{2h} (x_{2m}) + f_{2h} (x_{2m+2})}{2}, \qquad x_m=mh,
 \end{equation}
 and $R_h$ is the ``restriction'' operator that takes a function on the lattice $h\Z$ to a function on the lattice $2h\Z$, i.e.
\[ R_h f_h (x)=f_h (x)\ \mbox{for}\ x\in 2h\Z. \]
We will show that this discrete interpolator, which slightly averages the initial datum, will allow us to recover many of the dispersive properties which are common in continuous dispersive PDEs. More complex quadratic or higher order interpolations would also be admissible, but they are beyond the scope of this manuscript. Moreover, notice that $\Pi_h$ is the identity operator when applied to a piecewise linear function.

The continuous version of \eqref{eq:intro_discrete} was first studied in \cite[Theorem~1.2]{mypaper} and \cite[Theorem~2.1.3]{mythesis}. In particular, we showed that the memory effect gives rise to a loss of derivatives which can be overcome via the (Kato) smoothing effect, in a similar functional setting as the one proposed by Kenig, Ponce and Vega in the context of the KdV equation \cite{KPV}. 

Our first result is the well-posedness of the discrete equation \eqref{eq:intro_discrete} in the discrete version of this functional setting, see \Cref{sec:notation} and \Cref{sec:preliminaries0} for additional details regarding the notation.

\begin{thm}\label{thm:discretelwp} Let $\sigma=\frac{\alpha}{\beta}$, and suppose that
\begin{equation}\label{eq:parameterconditionsdisc}
2 > \frac{1}{\alpha}+\frac{1}{\beta}, \quad s\geq \frac{1}{2}-\frac{1}{2(p-1)},\quad \mbox{and}\quad \delta\in \left[s+\sigma-\alpha,\ \frac{\sigma}{2}-\frac{1}{2(p-1)}\right).
\end{equation}
where $\alpha\in (0,2)$ and $\beta\in (0,1)$. Then for every $f\in H^s(\R)$ there exists $T=T(\norm{f}_{H^s(\R)})>0$ (with $T(\rho)\rightarrow\infty$ as $\rho\rightarrow 0$) and a unique solution $u_h(t)$ to the integral equation associated to \eqref{eq:intro_discrete} satisfying
\begin{equation}\label{eq:class1}
u_h\in C([0,T],H^s_h),
\end{equation}
\begin{equation}\label{eq:class2}
\norm{\langle h^{-1} \nabla \rangle^{\delta} u_h }_{L^{\infty}_h L^2_T} < \infty ,
\end{equation}
\begin{equation}\label{eq:class3}
\norm{u_h}_{L^{2(p-1)}_h L^{\infty}_T} <\infty ,
\end{equation}
and 
\begin{equation}\label{eq:class4}
\norm{ \langle h^{-1} \nabla \rangle^{(s+\sigma-\alpha)/2} u_h}_{L^{4(p-1)}_h L^{4}_T} <\infty.
\end{equation}
Moreover, for $T'<T$, the map $f \mapsto u_h$ from $H^s(\R)$ to the space defined by \eqref{eq:class1}-\eqref{eq:class4} (with $T'$ instead of $T$) defined by solving the integral equation associated to \eqref{eq:intro_discrete} is locally Lipschitz.
\end{thm}
\begin{rk} The condition $2 > \frac{1}{\alpha}+\frac{1}{\beta}$ in \eqref{eq:parameterconditionsdisc} is necessary for the smoothing effect to overcome the loss of derivatives. Indeed, the smoothing effect allows us to gain $\frac{\sigma-1}{2}$ derivatives, while we have a loss of $\sigma-\alpha$ derivatives.
\end{rk}
\begin{rk} Note that the condition $2>\frac{1}{\alpha} + \frac{1}{\beta}$ together with $\beta\in (0,1)$ implies that $\alpha >1$ and $\beta > \frac{1}{2}$. Therefore we will consider only this range of parameters from now on.
\end{rk}

Once we have a solution to the continuous problem $u$, given by \cite[Theorem~1.2]{mypaper}, and a solution to the discrete problem $u_h$, given by \Cref{thm:discretelwp}, we may consider the linear interpolation of $u_h$. For $x\in [x_m,x_{m+1})$,
\begin{equation}\label{eq:intro_deflinearinterpolation}
p_h u_h (t,x):= u_h (t,x_m) + \frac{u_h (t,x_{m+1})-u_h (t,x_m)}{h} \cdot (x-x_m).
\end{equation}
In this way, both $u$ and $p_h u_h$ live in a common space $C([0,T],H^s_x(\R))$, and we may study the limit as $h\rightarrow 0$.

\begin{thm}[Continuum limit]\label{thm:continuum} Let $\alpha\in (1,2)$, $\beta\in (\frac{1}{2},1)$ and $\sigma=\frac{\alpha}{\beta}$. Consider the fractional Schr\"odinger equation
\begin{align*}
 \left\lbrace \begin{array}{ll}
 i^{\beta}\, \pa_t^{\beta} u = (-\Delta)^{\frac{\alpha}{2}} u \pm |u|^{p-1} u,\quad (t,x)\in [0,T]\times \R,\\
 u |_{t=0}=f,
 \end{array} \right.
\end{align*}
and the discrete model
\begin{align*}
 \left\lbrace \begin{array}{ll}
 i^{\beta}\, \pa_t^{\beta} u_h =  (-\Delta_h)^{\frac{\alpha}{2}} u_h \pm \Pi_h \, R_h \left( |u_h |^{p-1}\, u_h\right),\quad (t,x)\in [0,T]\times h\Z,\\
 u |_{t=0}=\Pi_h f_{2h}.
 \end{array} \right.
\end{align*}
Suppose that 
\[ 2 >\frac{1}{\alpha}+\frac{1}{\beta},\quad s=\frac{1}{2}-\frac{1}{2(p-1)},\quad \widetilde{s}:=\max\{s+\sigma-\alpha+, \frac{1}{2}+ \}<1,\]
and suppose that $f\in H^{\widetilde{s}}(\R)$. Then there exists a time $T>0$ such that both the solution to the continuous problem, $u$, and the solution to the discrete problem, $u_h$, exist, and 
\[p_h u_h \xrightarrow{\ h\rightarrow 0\ } u\] 
strongly in $L^{\infty}_T H^s_x$.
\end{thm}
\begin{rk} The condition $s+\sigma-\alpha<1$ is a byproduct of working with the linear interpolation $p_h u_h$, since the regularity of a piecewise linear function is limited. However, this could be removed by using a more sophisticated quadratic interpolation.
\end{rk}
\begin{rk} Let us give some intuition on the conditions for the parameters. Suppose that $p=3$, so that $s=\frac{1}{4}$. Once we fix $\alpha$, the range for $\beta$ follows from the conditions:
\[  2 >\frac{1}{\alpha}+\frac{1}{\beta}, \quad \mbox{and}\quad s+\sigma-\alpha<1,\]
For example, if $\alpha=3/2$ we obtain the range $\beta\in (\frac{3}{4},1)$, and thus $\sigma<2$.
As $\alpha$ decreases towards 1, the range for $\beta$ is reduced to a small neighborhood of 1. In other words, more dispersion allows for more memory.
\end{rk}
\begin{rk} We only consider the case $\alpha<2$ because the Fourier multiplier associated to the discrete Laplacian changes its behavior near zero as $\alpha$ passes the threshold $\alpha=2$. For $\alpha<2$ this multiplier behaves like $|\xi|^{\alpha}$ as $\xi\rightarrow 0$, while for $\alpha> 2$ its leading behavior is $|\xi|^2$ as $\xi\rightarrow 0$. Finally when $\alpha=2$, it behaves like $-(\log |\xi|)\, |\xi|^2$ as $\xi\rightarrow 0$.

This is the threshold that determines whether long-range interactions give rise to local or nonlocal behavior in the continuum limit, like in the work of Kirkpatrick et al$.$ \cite{KLS}. We expect that most of the techniques in this work can be extended to the case $\alpha\geq 2$, but the leading order of many estimates would be different and thus we decided to focus on the case $\alpha<2$.
\end{rk}

\subsection{Outline} 

In Section 2, we review some basic tools on the lattice and discuss the difficulties with dispersion and smoothing effect there. In Section 3, we prove well-posedness of the discrete equation on the lattice uniformly in the mesh-size. This proof requires some modifications compared to its continuous counterparts given the difficulties discussed in the previous section. Finally, in Section 4 we study the continuum limit and prove \Cref{thm:continuum}.

\subsection{Notation}\label{sec:notation} 

We write $A\lesssim B$ for $A\leq C B$, where the implicit constant $C$ might change from line to line. We will also write $A\lesssim_d B$ when the implicit constant, $C=C(d)$, depends on some variable $d$. We will often use the big $\O$ and little $o$ notation, e.g. $A=\O_d(B)$ when $A=\O(B)$ as $d\rightarrow 0$. 

We write $a-$ for a number $a-\varepsilon$, where $\varepsilon>0$ is small enough. Similarly, $a+$ means $a+\varepsilon$ for a small enough $\varepsilon>0$.

Given a function $f:[0,T]\times \R\rightarrow \C$ and $1\leq p,q\leq \infty$, we define
\[ \norm{f}_{L^q_T L^p_x}=\left(\int_0^T \left( \int_{\R} |f(t,x)|^p \, dx\right)^{q/p} \, dt\right)^{1/q} .\]
As usual, when $p=\infty$ or $q=\infty$ the norm will mean the essential supremum over the domain. When $f$ is defined for all times, i.e. $T=\infty$, we will write $\norm{f}_{L^q_t L^p_x}$ instead.

We use the following notation for the Fourier transform of a function $f:\R\rightarrow \C$:
\[ \widehat{f}(\xi)=\int_{\R} e^{-ix\xi} \, f(x)\, dx.\]
The inverse Fourier transform will be denoted by $f^{\vee}$. Given a function $p:\R\rightarrow\C$, we write $p(\nabla) f$ to denote the following Fourier multipler operator:
\[ \widehat{p(\nabla) f}(\xi) = p(\xi) \, \widehat{f}(\xi).\]
In particular, we often use the following Sobolev norms:
\begin{align*}
\norm{f}_{\dot{H}^s(\R)} & = \norm{ (-\Delta)^{\frac{s}{2}} f}_{L^2(\R)} = \norm{ |\nabla|^s f}_{L^2(\R)},\\
\norm{f}_{H^s(\R)} & = \norm{ \langle\nabla\rangle^s f}_{L^2(\R)}=\norm{ (1+ |\nabla|)^s f}_{L^2(\R)},
\end{align*}
The space $C([0,T],H^s(\R))$ is the space of continuous functions $f$ from $[0,T]$ to $H^s (\R)$ equipped with the norm $\max_{t\in [0,T]}\, \norm{f(t)}_{H^s(\R)}$.

Finally, we compile a list of some symbols commonly used throughout the paper:
\begin{itemize}
\item $\alpha$: number of space derivatives, given by Laplacian $(-\Delta)^{\frac{\alpha}{2}}$.
\item $\beta$: number of time derivatives, given by the Caputo derivative $\pa_t^{\beta}$.
\item $\sigma$: ratio $\alpha / \beta$.
\item $p$: degree of power-type nonlinearity.
\end{itemize}

\subsection*{Acknowledgements}

The author would like to thank Gigliola Staffilani for all her advice and encouragement, as well as Ethan Jaffe for several useful discussions and suggestions. 

\section{The linear equation}

\subsection{Definitions on the lattice}\label{sec:preliminaries0}

Consider the space $L^2_h:=\ell^2(h\Z)$ on the lattice, given by functions $u_h: h\Z\rightarrow \C$ such that
\[ \sum_{m\in\Z} |u_h (x_m)|^2 <\infty, \]
where $x_m=mh$ for $m\in\Z$. We define the inner product and norm
\[ (u_h,v_h)_{L^2_h}:= h \sum_{m\in\Z}u_h (x_m)\, \overline{v_h (x_m)} , \quad \norm{u_h}_{L^2_h}^2=(u_h,u_h)_{L^2_h}.\]
For a function $u_h: h\Z\rightarrow \C$ in $L^2_h$, we define its Fourier transform, $\widehat{u_h}: [-\pi,\pi] \rightarrow \C$, as follows:
\[ \widehat{u_h}(\xi)=\sum_{m\in \Z} u_h(x_m)\, e^{-i\xi m}.\]
Note that we choose this version of the Fourier transform, as opposed to taking $e^{-i\xi x_m}$, so that the Fourier transform of $u_h$ is defined in $[-\pi,\pi]$ regardless of $h$. With this definition, the Parseval identity yields
\[ (u_h,v_h)_{L^2_h}=h \int_{-\pi}^{\pi} \widehat{u_h}(\xi)\,\overline{\widehat{v_h}(\xi)}\, d\xi,\]
and the following inversion formula holds:
\[ u_h(x_m)=\frac{1}{\sqrt{2\pi}} \, \int_{-\pi}^{\pi} \widehat{u_h}(\xi) \, e^{im\xi}\, d\xi.\]

For a function $u_h: h\Z\rightarrow \C$ in $L^2_h$, we define discrete fractional Laplacian on the lattice $h\Z$ as
\begin{equation}\label{eq:defDiscreteLaplacian}
 (-\Delta_h)^{\alpha /2} u_h(x_m) =h\, \sum_{n\neq m} \frac{u_h(x_m)-u_h(x_n)}{|x_m-x_n|^{1+\alpha}}.
 \end{equation}
Note that the Fourier transform of \eqref{eq:defDiscreteLaplacian} is
\[ h^{-\alpha}\,\sum_{m\in\Z} \sum_{n\neq m} \frac{u_h(x_m)-u_h(x_n)}{|m-n|^{1+\alpha}} \, e^{-i\xi m}= h^{-\alpha}\,w(\xi)\, \widehat{u_h}(\xi),\]
where
\begin{equation}\label{eq:defw}
 w(\xi)=\sum_{n\neq m} \frac{1-e^{-i\xi (m-n)}}{|m-n|^{1+\alpha}}= 2\,\sum_{n=0}^{\infty} \frac{1-\cos \xi n}{|n|^{1+\alpha}}\geq 0 .
 \end{equation}

Similarly, we define the $H^s_h$ norm on the lattice as follows:
\[ \norm{u_h}^2_{H^s_h}:= h \int_{-\pi}^{\pi} (1+h^{-2s} |\xi|^{2s})\, |\widehat{u_h}(\xi)|^2 \, d\xi.\]
As explained in \cite{KLS}, it is easy to show that this norm is equivalent to the one given by the inner product defined before:
\[ \norm{u_h}_{L^2_h}+ \norm{(-\Delta_h)^{s/2} u_h}_{L^2_h}.\]
The normed space $\dot{H}_h^s$ is defined analogously. For $s=1$, we have one more useful equivalent norm for $H^{1}_h$, given by
\[ \norm{u_h}_{L^2_h}+ \norm{D_{h}^{+} u_h}_{L^2_h},\]
where $D_{h}^{+}$ is a forward difference:
\begin{equation}\label{eq:defForwardDifference}
D_{h}^{+} u_h (x_m)= \frac{u_h (x_{m+1})- u_h (x_m)}{h}.
\end{equation}

More generally, we define the space $L^p_h$ (which agrees with $\ell^p(h\Z)$ with additional scaling) for $1\leq p<\infty$ as the space of functions $u_h: h\Z\rightarrow \C$ such that
\[ \norm{u_h}_{L^p_h}:= \left( h \, \sum_{m\in\Z} |u_h (x_m)|^p\right)^{1/p} <\infty. \]
In the case $p=\infty$, we take the norm given by the supremum.

\begin{rk} In \cite{KLS},  a more general discretization of the Laplacian is considered. In particular, they define 
\[ \mathcal{L}^{J}_h u_h (x_m):= h\, \sum_{n\neq m} J_{n-m} \, \left[u_h(x_m)-u_h(x_n)\right].\]
The coefficients $\{J_n\}_{|n|\geq 1}$, which account for long-range interactions, must satisfy the following conditions:
\begin{itemize}
\item $J_1>0$,
\item $J_n=J_{-n}\geq 0$ for all $|n|\geq 1$, and
\item $\lim_{|n|\rightarrow\infty} |x_n|^{1+\alpha}\, J_n = C_{\alpha}>0$.
\end{itemize}
As a consequence, $w$ in \eqref{eq:defw} would also depend on $\{J_n\}_{|n|\geq 1}$, but their results only depend on the asymptotic behavior of this sequence. Unfortunately, we cannot yet handle such a general situation, since we will need to exploit deeper properties of $w$, such as knowledge about the zeroes of $w'$ and $w''$. This requires a more careful analysis that we can't yet carry out for a general kernel.
\end{rk}

We define the discretization of a function $f:\R \rightarrow \C$ as follows:
\begin{equation}\label{eq:defdiscretization}
f_h(x_m)=\frac{1}{h}\,\int_{x_m}^{x_{m+1}} f (x) \, dx, \qquad \mbox{for}\ x_m\in h\Z.
\end{equation}
Then we have the following result in \cite[Lemma~3.6]{KLS}:
\begin{prop}\label{thm:uniformh}
Suppose that $f\in H^s(\R)$ (resp. $\dot{H}^s(\R)$) for some $0\leq s\leq 1$. Then we have
\begin{align*}
 \norm{f_h}_{H^s_h} & \lesssim \norm{f}_{H^s(\R)},\\
  \norm{f_h}_{\dot{H}^s_h} & \lesssim \norm{f}_{\dot{H}^s(\R)},
\end{align*}
where the implicit constants are independent of $h$.
\end{prop}
The proof is an application of complex interpolation between $s=0$ (straight-forward) and $s=1$, which is based on \eqref{eq:defForwardDifference}.

Another important result is the discrete analog of the Sobolev embedding theorem \cite[Lemma~ 3.1]{KLS}:
\begin{lem}[Discrete Sobolev inequality]
For every $\frac{1}{2}< s\leq 1$, there exists a constant $C=C(s)>0$ independent of $h>0$ such that 
\[ \norm{u_h}_{L^{\infty}_h} \leq C \, \norm{u_h}_{H^{s}_h}\]
for all $u_h\in L^2_h$.
\end{lem}

In the following lemma, we summarize some useful facts about the function $w$ defined in \eqref{eq:defw}. Some of these results follow from properties of the polylogarithm, but we give the proof for completeness.

\begin{lem}\label{thm:aboutw} Let $\alpha\in (1,2)$. Then the function 
\[ w(\xi)= 2\,\sum_{n=1}^{\infty} \frac{1-\cos \xi n}{n^{1+\alpha}}\geq 0\]
has the following properties:
\begin{enumerate}[(i)]
\item There exist constants $c_1,c_2>0$ such that 
\[c_1 |\xi|^{\alpha} \leq w(\xi)\leq c_2 |\xi|^{\alpha} \qquad \xi\in [0,\pi].\]
\item $w$ is one-to-one on the interval $[0,\pi]$.
\item $w$ is differentiable and $w'(\xi)=\O(|\xi|^{\alpha-1})$ as $\xi\rightarrow 0$.
\item $w'(\xi)>0$ for every $\xi\in (0,\pi)$. 
\item $w(\xi)=c |\xi|^{\alpha} + \O (|\xi|^2 )$ as $\xi\rightarrow 0$.
\item There exist some $c_1,c_2>0$ such that 
\[ c_1 \, |\xi|^{\alpha-1}\, (\pi-\xi)\leq w'(\xi)\leq c_2\, |\xi|^{\alpha-1}\] 
for every $\xi\in (0,\pi)$.
\item $w'$ is differentiable in $(0,\pi]$, and $w''(\xi)=\O (|\xi|^{\alpha-2})$ as $\xi\rightarrow 0$.
\item $w''(\xi)$ is monotone decreasing on $(0,\pi)$ and has a unique zero at $\xi_0 \in (0,\frac{\pi}{2})$.
\end{enumerate}
\end{lem}
\begin{rk} From now on, we will assume that we normalize $w$ so that $c=1$ in part $(v)$.
\end{rk}
\begin{proof}
\begin{enumerate}[(i)]
\item See Appendix A in \cite{KLS} or the proof for $w'$ in step (iii) below, which is analogous.
\item It follows from the previous step and part (iv) below.
\item It is easy to prove that
\[ w'(\xi)=2\,\sum_{n=1}^{\infty}\frac{\sin \xi n}{n^{\alpha}}.\]
Then we can write
\[ \frac{w'(\xi)}{\xi^{\alpha-1}}= 2\,\sum_{n=1}^{\infty} \frac{\xi}{(\xi n)^{\alpha}}\, \sin \xi n.\]
Therefore,
\[ \lim_{\xi\rightarrow 0} \frac{w'(\xi)}{\xi^{\alpha-1}} = 2\, \int_{0}^{\infty} \frac{\sin y}{y^{\alpha}} \, dy=:c_{\alpha}=\frac{\pi}{2\,\Gamma (\alpha)\, \sin \left(\frac{\alpha \pi}{2}\right)}>0.\]
This can be found in \cite{KLS}. See \cite{dickinson} for a careful proof ot the first equality. 
\item Note that 
\begin{equation}\label{eq:polylog}
 w'(\xi)=\frac{\mbox{Li}_{\alpha}(e^{i\xi}) - \mbox{Li}_{\alpha}(e^{-i\xi})}{2i},
 \end{equation}
where $\mbox{Li}_{\alpha}(z)$ stands for the polylogarithm. This special function admits the following representation
\[ \mbox{Li}_{\alpha}(z)= \frac{1}{\Gamma(\alpha)} \int_0^{\infty} \frac{y^{\alpha-1}}{e^{y}/z -1}\, dy \]
for $z\in \C$ (except when $z$ is real and $z\geq 1$).
Using this, one can show that 
\begin{equation}\label{eq:polylog2} 
w'(\xi)= \frac{\sin\xi}{\Gamma(\alpha)} \int_0^{\infty} \frac{y^{\alpha-1} e^{y}}{e^{2y}-2 \cos\xi +1}\, dy .
\end{equation}
For a fixed $\xi\in (0,\pi)$, $\sin\xi$ is positive and the integrand is always positive and integrable.
One may even prove that $w'(\xi)=\O(|\xi|^{\alpha-1})$ as $\xi\rightarrow 0$ from this formula.
\item For small $\xi>0$, we divide the integration in \eqref{eq:polylog2} over three subintervals: $[0,\xi]$, $[\xi,1]$ and $[1,\infty)$. It is easy to check that the first two give a term in $\O (|\xi|^{\alpha-2})$, while the last integral gives $\O (1)$. After multiplication by $\sin \xi \sim \xi$, we obtain that 
\begin{equation}\label{eq:derw}
 w'(\xi)= C_{\alpha} \,|\xi|^{\alpha-1} + \O (|\xi|),
 \end{equation}
which yields the desired expansion for $w(\xi)$ upon integration.
\item It follows from the fact that $w'$ is continuous, $w'(0)=w'(\pi)=0$, $w'(\xi)\geq 0$ and $w'(\xi)=\O(|\xi|^{\alpha-1})$ near $\xi=0$. The behavior at $\xi=\pi$ follows from the factor $\sin\xi$ in \eqref{eq:polylog2}.
\item From \eqref{eq:polylog}, it is enough to show that $\mbox{Li}_{\alpha}(e^{i\xi})$ is differentiable for $\xi\in (0,\pi]$, and its derivative is $i\, \mbox{Li}_{\alpha-1}(e^{i\xi})$. Recall that 
\[ \mbox{Li}_{\alpha}(e^{i\xi})= \frac{1}{\Gamma(\alpha)} \int_0^{\infty} \frac{e^{i\xi}\,y^{\alpha-1}}{e^{y} -e^{i\xi}}\, dy =: \int_0^{\infty} f(\xi,y)\, dy.\]
By the Dominated Convergence theorem, we only need to show that $f$ is differentiable and that $|\pa_{\xi} f(\xi,y)|\leq F(y)\in L^1(dy)$ for a.e. $\xi$. Fix $\xi_0\in (0,\pi]$ and consider a neighborhood $\xi\in (\xi_0-\varepsilon,\xi_0+\varepsilon)$ for $\varepsilon>0$ small enough. Then the denominator of $f$ is bounded away from zero and the function is differentiable. In particular 
 \begin{align*}
  \pa_{\xi} f(\xi,y) & = \frac{1}{\Gamma(\alpha)}\,\frac{i\,e^{i\xi}\,y^{\alpha-1}}{e^{y} -e^{i\xi}} + 
  \frac{1}{\Gamma(\alpha)}\,\frac{i\,e^{2i\xi}\,y^{\alpha-1}}{(e^{y} -e^{i\xi})^2} \\
  &  = \frac{1}{\Gamma(\alpha)}\,\frac{i\,e^{i\xi}\,(e^{y} -e^{i\xi})\,y^{\alpha-1}}{(e^{y} -e^{i\xi})^2} + 
  \frac{1}{\Gamma(\alpha)}\,\frac{i\,e^{2i\xi}\,y^{\alpha-1}}{(e^{y} -e^{i\xi})^2}  = \frac{1}{\Gamma(\alpha)}\,\frac{i\,e^{i\xi}\,e^{y}\,y^{\alpha-1}}{(e^{y} -e^{i\xi})^2}.
  \end{align*}
Note that
\[ |\pa_{\xi} f(\xi,y)|\lesssim \frac{y^{\alpha-1}}{e^{y} \, \inf_{\xi\in B(\xi_0,\varepsilon)}\, |1-e^{i\xi-y}|} \in L^1(dy).\]

Finally, integration by parts yields:
\begin{align*}
 \frac{d}{d\xi} \mbox{Li}_{\alpha}(e^{i\xi}) & = \int_0^{\infty} \frac{1}{\Gamma(\alpha)}\,\frac{i\,e^{i\xi}\,e^{y}\,y^{\alpha-1}}{(e^{y} -e^{i\xi})^2}\, dy = \frac{i}{\Gamma(\alpha-1)} \, \int_0^{\infty} \frac{e^{i\xi}\, y^{\alpha-2}}{e^{y}-e^{i\xi}}\, dy = i\, \mbox{Li}_{\alpha-1}(e^{i\xi}).
\end{align*}
This integral representation gives the bound $\mbox{Li}_{\alpha-1}(e^{i\xi})=\O (|\xi|^{\alpha-2})$ as $\xi\rightarrow 0$, which follows for $w''$ thanks to the identity
\[ w''(\xi)= \frac{1}{2}\, \mbox{Li}_{\alpha-1}(e^{i\xi}) +\frac{1}{2}\, \mbox{Li}_{\alpha-1}(e^{-i\xi}).\]
\item From this equality, we may write
\[ w''(\xi)=\frac{1}{\Gamma (\alpha-1)} \, \int_0^{\infty} y^{\alpha-2} \, \frac{e^{y} \cos\xi-1}{e^{2y}-2 e^y \cos\xi + 1} \, dy.\]
Note that $\alpha>1$ is critical for local integrability around zero.  

From the previous step, we know that $\lim_{\xi\rightarrow 0+} w''(\xi)=+\infty$. Note also that $w''(\xi)<0$ for $\xi\in \left[ \frac{\pi}{2},\pi\right)$, since the integrand will be negative. Therefore there exists at least one point $\xi_0<\frac{\pi}{2}$ such that $w''(\xi_0)=0$. We want to show that $w''$ is monotone decreasing, and thus this point is unique.

To do that, we show that the integrand is monotone decreasing in $\xi$. The derivative with respect to $\xi$ of the integrand is
\[ y^{\alpha-2} \, \sin\xi \, \frac{e^{y}-e^{3y}}{(e^{2y}-2e^{y}\cos\xi+1)^2} \leq 0 ,\]
which concludes the proof. Note that the derivative of the integrand might not be integrable itself.
\end{enumerate}
\end{proof}

\subsection{Discrete linear equation}\label{sec:preliminaries}

We now study a natural generalization of the linear continuous equation. For $f\in H^s(\R)$, $0\leq s\leq 1$ consider the problem
\begin{align}\label{eq:discretelinear}
 \left\lbrace \begin{array}{ll}
 i^{\beta}\, \pa_t^{\beta} u_h = (-\Delta_h)^{\frac{\alpha}{2}} u_h,\quad (t,x)\in [0,T]\times h\Z,\\
 u_h |_{t=0}=f_h ,
 \end{array} \right.
\end{align}
where $f_h$ is the discretization of $f$ as defined in \eqref{eq:defdiscretization}.

As explained in \cite{mypaper}, one may take the Fourier transform in space and Laplace transform in time, and obtain the following representation for the solution:
\[ \widehat{u_h}(t,\xi)=\widehat{f_h}(\xi)\,\sum_{k=0}^{\infty} \frac{i^{-\beta k} t^{\beta k} h^{-\alpha k}\, w(\xi)^{k}}{\Gamma (k\beta +1)}=E_{\beta}(i^{-\beta}t^{\beta}h^{-\alpha}w(\xi)) \widehat{f_h}(\xi).\]
As in the continuous case, our Fourier multiplier is given by the \emph{Mittag-Leffler function}:
\begin{equation}\label{eq:defML}
 E_{\beta}(z)=\sum_{k=0}^{\infty} \frac{z^k}{\Gamma (k\beta +1)}.
 \end{equation}
This function is an entire function in the complex plane. More details about this derivation may be found in \cite{kai}, and also \cite[Appendix~A]{mypaper}. We will write the solution of the linear equation as follows:
\begin{equation}\label{eq:defL}
u_h(t)=L_{t,h} f_h:= \left( E_{\beta}(i^{-\beta}t^{\beta}h^{-\alpha}w(\cdot)) \widehat{f_h}(\cdot)\right)^{\vee},
\end{equation}
where ${}^{\vee}$ denotes the inverse Fourier transform.

The Mittag-Leffler function enjoys the following asymptotics:
\begin{equation}\label{eq:MLasymp}
 E_{\beta}(z)= \frac{1}{\beta}\, e^{z^{1/\beta}} + \sum_{k=1}^{N-1}\frac{z^k}{\Gamma (1-\beta k)}+ \O(|z|^{-N} ),\  \mbox{as}\ |z|\rightarrow\infty.
 \end{equation}
This is valid when $|\arg(z)| \leq \frac{\beta\pi}{2}$ and for any integer $N\geq 2$. See \cite[Chapter~18]{bate} for more information.

Using \eqref{eq:MLasymp} one can show that our Fourier multiplier is uniformly bounded and therefore the solution to \eqref{eq:discretelinear} satisfies:
\[ \norm{L_{t,h} f_h}_{H^s_h}\lesssim \norm{f_h}_{H^s_h}\lesssim \norm{f}_{H^s(\R)} \]
where the implicit constants are independent of $t$ and $h$.

Consider now the inhomogeneous equation:
\begin{align}\label{eq:inhomogeneous}
 \left\lbrace \begin{array}{ll}
 i^{\beta}\, \pa_t^{\beta} u_h= (-\Delta_h)^{\frac{\alpha}{2}} u_h  + g_h,\quad (t,x)\in [0,T]\times h\Z,\\
 u_h |_{t=0}=f_h.
 \end{array} \right.
\end{align}
where we will later set $g_h$ to be a power-type nonlinearity depending on $u_h$. By using a fractional generalization of the Duhamel formula, we can write the solution to \eqref{eq:inhomogeneous} as
\begin{equation}\label{eq:discreteDuhamel}
\begin{split}
u_h(t,x_m)=&\ \frac{1}{\sqrt{2\pi}}\, \int_{-\pi}^{\pi} E_{\beta}(i^{-\beta}t^{\beta}h^{-\alpha}w(\xi))\, \widehat{f_h}(\xi)\, e^{im\xi}\, d\xi \\
& +  \frac{i^{-\beta}}{\sqrt{2\pi}}\, \int_0^{t}\int_{-\pi}^{\pi} (t-t')^{\beta-1}\, E_{\beta,\beta} (i^{-\beta}(t-t')^{\beta}h^{-\alpha}w(\xi))\, \widehat{g_h}(t',\xi)\, e^{im\xi}\, d\xi \, dt',
\end{split}
\end{equation}
where
\begin{equation}\label{eq:defgenML}
 E_{\beta,\beta}(z)=\sum_{k=0}^{\infty} \frac{z^k}{\Gamma (k\beta +\beta)}
 \end{equation}
is the \emph{generalized Mittag-Leffler} function, which is also entire in the complex plane. The asymptotics for this function are as follows:
\begin{equation}\label{eq:genMLasymp}
\begin{split}
t^{\beta-1}\, E_{\beta,\beta}(i^{-\beta}t^{\beta} h^{-\alpha} w(\xi)) = & \, \frac{1}{\beta}\,i^{\beta-1}\,(h^{-\alpha} w(\xi))^{\frac{1-\beta}{\beta}}\, e^{-it(h^{-\alpha} w(\xi))^{1/\beta}} \\ 
& + \sum_{k=2}^N \frac{\Gamma (\beta k -\beta)^{-1}}{t^{1+(k-1)\beta}\,h^{-k\alpha}\, w(\xi)^k} + \O \left( t^{-1-N\beta}\,h^{(N+1)\alpha}\, w(\xi)^{-N-1} \right)
\end{split}
 \end{equation}
 as $t^{\beta} h^{\alpha} |w(\xi)|\rightarrow \infty$. This is valid for any integer $N\geq 2$. Let us set
 \begin{equation}\label{eq:defphasefunction}
 \phi_h(\xi):=h^{-\sigma} w(\xi)^{1/\beta}
 \end{equation}
 so that the leading terms in \eqref{eq:MLasymp} and \eqref{eq:genMLasymp} are $e^{-it\phi_h(\xi)}$ and $\phi_h(\xi)^{1-\beta}\,e^{-it\phi_h(\xi)}$ respectively.

Our first goal would be to show local well-posedness of the initial value problem (IVP) given by \eqref{eq:discreteDuhamel}. The main problem is that we are losing
derivatives in our basic $L^{\infty}_t L^2_h$-estimate. Indeed, note that by \Cref{thm:aboutw} the leading order in \eqref{eq:genMLasymp} is of size
\[ h^{\alpha- \frac{\alpha}{\beta}}\, w(\xi)^{\frac{1}{\beta}-1} \sim \Big |\frac{\xi}{h}\Big |^{\frac{\alpha}{\beta}-\alpha}, \]
where the exponent is positive because $\beta<1$. A way to overcome this loss of derivatives is presented in \cite{mypaper} for the continuous equation. In that setting, the idea is to exploit some smoothing effect by working in the space $X_T^s \subset C_t([0,T],H^s_x(\R))$ of functions with finite smoothing and maximal norms, i.e. 
\[ \norm{ \langle \nabla\rangle^{\delta} \ \cdot\ }_{L^{\infty}_x (\R , L^2_t ([0,T]) )}\quad \mbox{and}\quad \norm{\ \cdot\ }_{L^{2(p-1)}_x (\R , L^{\infty}_t ([0,T]) )}.\] 
Under some technical conditions on the parameters, a fixed point argument yields an interval of existence $[0,T]$ where $T$ only depends on an inverse power of the $H^s_x(\R)$-norm of the initial data. We give the full theorem below for completeness, which may be found in \cite[Theorem~1.2]{mypaper} and \cite[Chapter~2]{mythesis}.

\begin{thm}\label{thm:contLWP} Consider the space-time fractional nonlinear Schr\"odinger initial value problem:
\begin{equation}\label{eq:contequation}
\left\{ \begin{array}{ll}
i^{\beta} \pa_t^{\beta} u  = (-\Delta_x)^{\frac{\alpha}{2}}\, u \pm |u|^{p-1} u \qquad (t,x)\in (0,\infty)\times\R , \\
u |_{t=0}  = f \in H^s (\R),
\end{array}\right.
\end{equation}
for some odd integer $p\geq 3$, $\alpha>0$ and $\beta\in (0,1)$. With $\sigma=\frac{\alpha}{\beta}$, suppose that 
\begin{equation}\label{eq:parameterconditionscont}
2 > \frac{1}{\alpha}+\frac{1}{\beta}, \quad s\geq \frac{1}{2}-\frac{1}{2(p-1)},\quad \mbox{and}\quad \delta\in \left[s+\sigma-\alpha, \frac{\sigma}{2}-\frac{1}{2(p-1)}\right).
\end{equation}
Then for every $f\in H^s(\R)$ there exists $T=T(\norm{f}_{H^s(\R)})>0$ (with $T(\rho)\rightarrow\infty$ as $\rho\rightarrow 0$) and a unique solution $u(t,x)$ to the integral equation associated to \eqref{eq:contequation}, satisfying
\begin{equation}\label{class1}
u\in C([0,T],H^s (\R)),
\end{equation}
\begin{equation}\label{class2}
\norm{\langle \nabla \rangle^{\delta} u }_{L^{\infty}_x L^2_T} < \infty ,
\end{equation}
\begin{equation}\label{class3}
\norm{u}_{L^{2(p-1)}_x L^{\infty}_T} <\infty ,
\end{equation}
and 
\begin{equation}\label{class4}
\norm{\langle \nabla \rangle^{(s+\sigma-\alpha)/2} u }_{L^{4(p-1)}_x L^4_T} < \infty .
\end{equation}
Moreover, for any $T'\in (0,T)$ there exists a neighborhood $V$ of $f$ in $H^s (\R)$ such that the map $\tilde{f} \rightarrow \tilde{u}$ from $V$ into the class defined by \eqref{class1}-\eqref{class4} with $T'$ instead of $T$ is Lipschitz. 
\end{thm}

\Cref{thm:contLWP} follows from a fixed point theorem on a ball in the space given by \eqref{class1}-\eqref{class4}. This choice of space follows from the following key linear estimates, which may be found in \cite[Section~2]{mypaper} and \cite[Section~2.2]{mythesis}:

\begin{prop}\label{thm:main_cont_est} Let $\alpha\in (1,2)$, $\beta\in (\frac{1}{2},1)$, $\sigma=\alpha/\beta$,  $\gamma=\frac{\sigma-1}{2}$ and $\tilde{\gamma}= \alpha - \frac{\sigma+1}{2}$. Let $L_t$ and $N_t$ be the linear and nonlinear flow associated to \eqref{eq:contequation}, i.e.
\[
\begin{split}
L_t f & =\left(  E_{\beta}(i^{-\beta} t^{\beta} |\cdot |^{\alpha}) \widehat{f}\right)^{\vee},\\
N_t g & =\left(  t^{\beta-1}\, E_{\beta,\beta}(i^{-\beta} t^{\beta} |\cdot |^{\alpha}) \widehat{g}\right)^{\vee}.\\
\end{split}
\]
Let $s=\frac{1}{2}-\frac{1}{p}$, for $4\leq p<\infty$. Then for any $0\leq \gamma'<\gamma$ and any $0\leq \tilde{\gamma}'<\tilde{\gamma}$ we have:
\[
\begin{split}
\norm{L_t f}_{L^{\infty}_t L^2_x} & \lesssim \norm{f}_{L^2_x}\\
\norm{\langle \nabla\rangle^{\gamma'} L_t f}_{L^{\infty}_x L^2_T} & \lesssim \langle T\rangle^{1/2}\, \norm{f}_{L^2_x}\\
\norm{L_t f}_{L^p_x L^{\infty}_T} & \lesssim \norm{\langle\nabla\rangle^s f}_{L^2_x}\\
\norm{L_t f}_{L^{2p}_x L^{4}_T} & \lesssim \norm{\langle\nabla\rangle^{\frac{s-\gamma'}{2}} f}_{L^2_x}\\
\end{split}
\]
Moreover,
\[
\begin{split}
\norm{\int_0^t N_{t-t'} g (t',x)\, dt'}_{L^2_x(\R)} & \lesssim \langle T\rangle^{\beta-1/2}\, \norm{g}_{L^2_T L^2_x}\\
\norm{\langle\nabla\rangle^{\tilde{\gamma}'} \, \int_0^t N_{t-t'} g (t',x)\, dt'}_{L^{\infty}_x L^2_T} & \lesssim \langle T\rangle^{1/2}\, \norm{g}_{L^1_T L^2_x}\\
\norm{\int_0^t N_{t-t'} g (t',x)\, dt'}_{L^p_x L^{\infty}_T} & \lesssim \langle T\rangle^{\beta-1/2}\, \norm{\langle\nabla\rangle^{s+\sigma-\alpha} g}_{L^1_T L^2_x}\\
\norm{\int_0^t N_{t-t'} g (t',x)\, dt'}_{L^{2p}_x L^{4}_T} & \lesssim\langle T\rangle\, \norm{\langle\nabla\rangle^{\frac{s+\sigma-\alpha-\tilde{\gamma}'}{2}} g}_{L^2_T L^2_x}\\
\end{split}
\]
\end{prop}

Given \Cref{thm:contLWP}, a reasonable idea would be to work in the discrete analog of the space $X_T^s$ above, and prove that the IVP given by \eqref{eq:discreteDuhamel} is locally well-posed. However, this is not possible because the smoothing effect is not readily available in the discrete setting, as exemplified by \cite[Theorem~2.2]{zuazua}:
\begin{thm}\label{thm:zuazua}
Let $t>0$ and $s>0$. Consider
\[ \widetilde{\Delta_h} \varphi_h (x_m):= \frac{\varphi_h (x_{m+1}) - 2\varphi_h (x_m) + \varphi_h (x_{m-1})}{h^2} \]
and let $\widetilde{L}_{h,t} \varphi_h$ be the solution to the following IVP:
\[  \left\lbrace \begin{array}{ll}
 i\, \pa_t u_h(t,x) + \widetilde{\Delta_h} u_h (t,x)=0,\quad (t,x)\in [0,T]\times h\Z,\\
 u_h \mid_{t=0} =\varphi_h.
 \end{array} \right.\]
Then
\begin{equation}\label{eq:zuazua1}
\sup_{h>0,\ \varphi_h\in L^2_h} \frac{h \sum_{|x_m|\leq 1} |(-\widetilde{\Delta_h})^{s/2} \widetilde{L}_{h,t} \varphi_h (x_m)|^2}{\norm{\varphi_h}_{L^2_h}^2} = \infty ,
\end{equation}
and 
\begin{equation}\label{eq:zuazua2}
\sup_{h>0,\ \varphi_h\in L^2_h} \frac{h \sum_{|x_m|\leq 1} \int_0^t |(-\widetilde{\Delta_h})^{s/2} \widetilde{L}_{h,t'} \varphi_h (x_m)|^2 \, dt' }{\norm{\varphi_h}_{L^2_h}^2} = \infty .
\end{equation}
\end{thm}
\begin{rk}
Note that our operators $L_{t,h}$ and $\Delta_h$ are not quite the same as those in this result. However, the reason behind this theorem still applies to our setting. Indeed, \Cref{thm:zuazua} is still true if in \eqref{eq:zuazua1}-\eqref{eq:zuazua2} we take the supremum over those $\varphi_h\in L^2_h$ supported on a set that contains at least one critical point of the Fourier multiplier associated to $\widetilde{\Delta_h}$.
\end{rk}
\begin{rk} 
A similar result shows that there are no uniform-in-$h$ dispersive estimates nor Strichartz estimates. Indeed for any $t>0$ and $r>r_0\geq 1$,
\begin{equation*}
\sup_{h>0,\, \varphi_h\in L^{r_0}_h} \frac{\norm{\widetilde{L}_{h,t} \varphi_h}_{L^r_h}}{\norm{\varphi_h}_{L^{r_0}_h}} = \infty , \quad \mbox{and}\quad 
\sup_{h>0,\, \varphi_h\in L^{r_0}_h} \frac{\norm{\widetilde{L}_{h,t} \varphi_h}_{L^1([0,T],L^r_h)}}{\norm{f_h}_{L^{\varphi_h}_h}} = \infty .
\end{equation*}
In this case, the problematic points are the points of inflection of the Fourier multiplier associated to $\widetilde{\Delta_h}$, and the result remains true if we only take the supremum over those $\varphi_h\in  L^{r_0}_h$ supported on a set that contains at least one such point.
\end{rk}

The intuition behind this phenomenon is the following: in the continuous setting, the fractional Laplacian $(-\Delta_x)^{\frac{\alpha}{2}}$ corresponds to the Fourier multiplier $|\xi|^{\alpha}$. In the discrete setting, we try to approximate this on $[-\frac{\pi}{h},\frac{\pi}{h}]$ by $h^{-\alpha} w(h\xi)$ (as the mesh-size $h$ tends to zero). These two functions have similar behavior near zero, as shown in \Cref{thm:aboutw}. However, $h^{-\alpha} w(h\xi)$ has critical points that are not present in $|\xi|^{\alpha}$, see \Cref{fig:phi}.

\begin{figure}[h]
\centering
\begin{subfigure}
  \centering
  \includegraphics[scale=0.25]{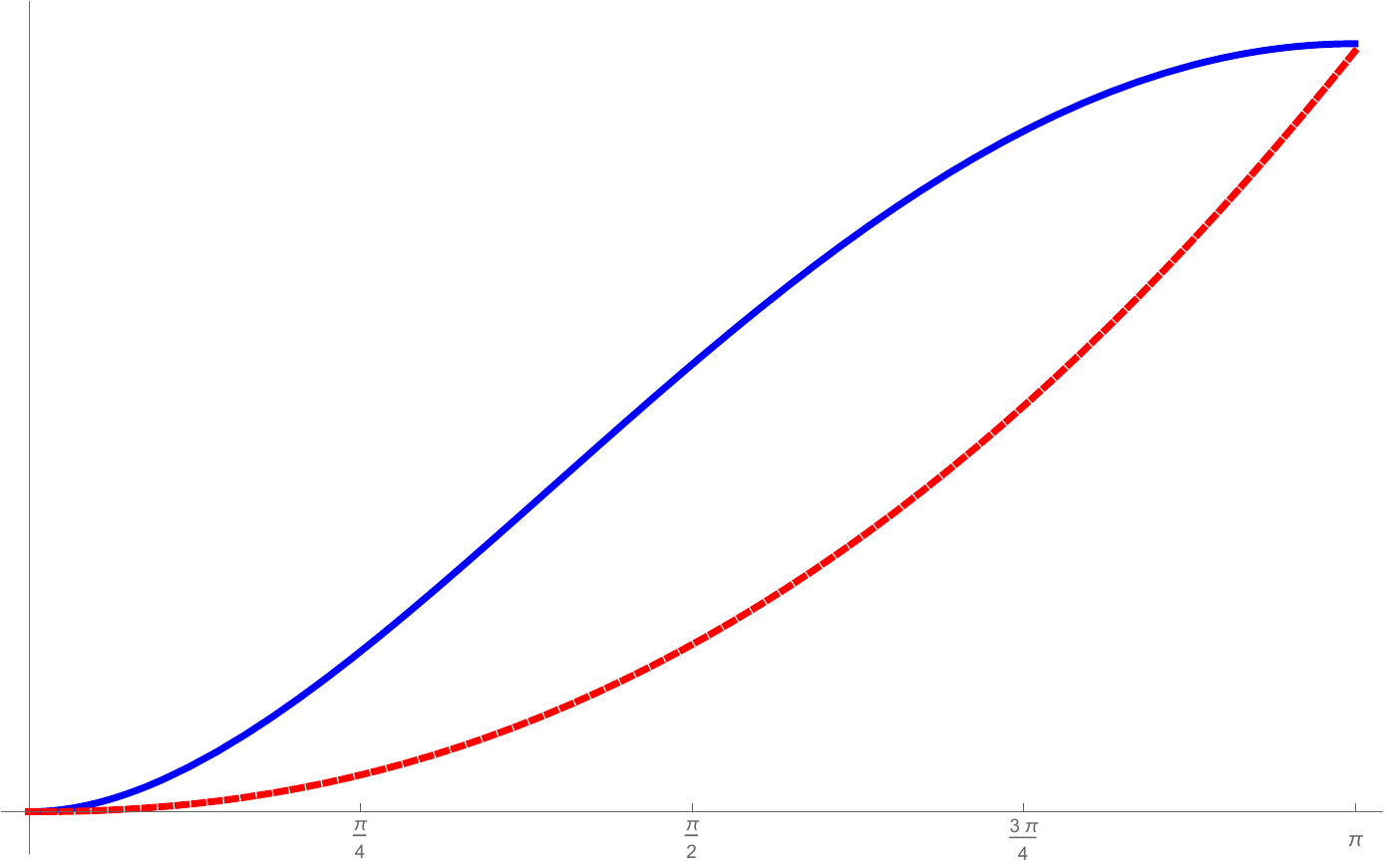}
\end{subfigure}% 
\hfill
\begin{subfigure}
  \centering
  \includegraphics[scale=0.3]{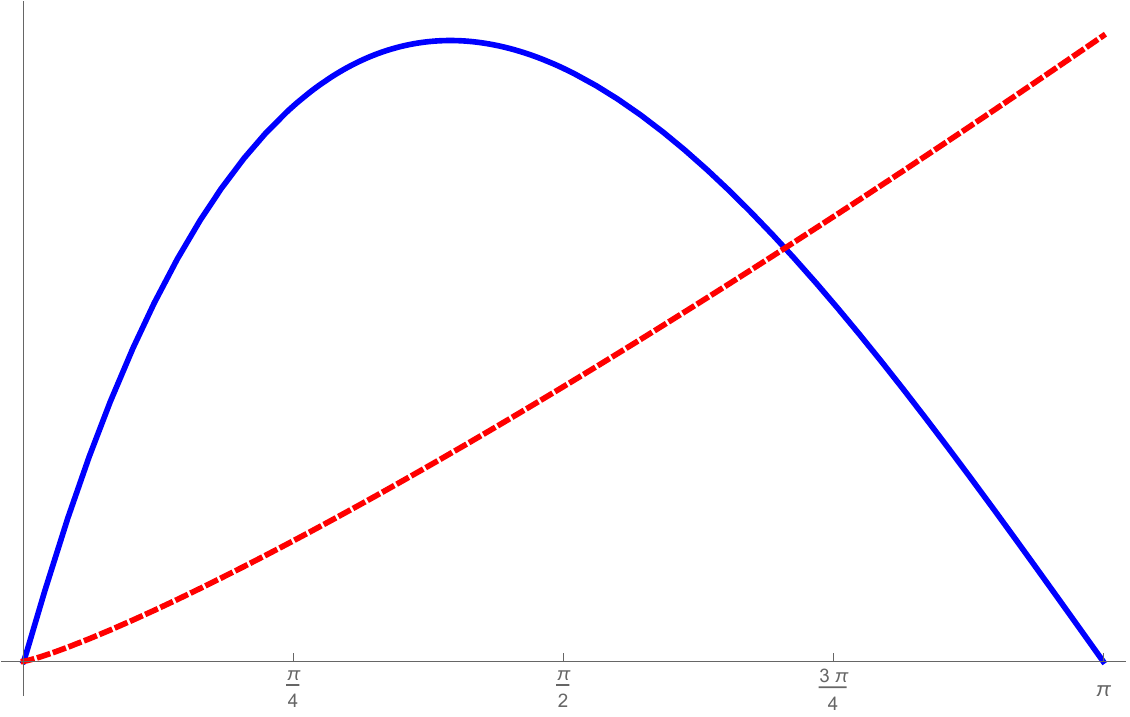}
\end{subfigure}
\caption{\small{Graph of $w(\xi)^{1/\beta}$  (blue) and $|\xi|^{\alpha/\beta}$ (red) for $\alpha=1.75$ and $\beta=0.8$ (left), and their derivatives (right). Produced with Wolfram Mathematica.}}
\label{fig:phi}
\end{figure} 

By taking pathological initial data supported in those critical points, one can obtain a result such as the one in \eqref{eq:zuazua1}-\eqref{eq:zuazua2}. However, those critical points are not present in the continuous setting, which suggests that such a discrete model cannot be expected to capture the continuous behavior. As proposed in \cite{zuazua}, one can get around this issue by filtering the initial data.

\subsection{Filtering initial data}

For a function $f_{2h}\in L^2_{2h}$, define the discrete interpolation operator $\Pi_h: L^2_{2h} \rightarrow L^2_{h}$ as follows:
\begin{align}\label{eq:definitionfilter}
 \Pi_h f_{2h} (x_{2m}) & = f_{2h}(x_{2m}),\\
  \Pi_h f_{2h}(x_{2m+1}) & = \frac{f_{2h}(x_{2m})+f_{2h}(x_{2m+2})}{2},\nonumber
\end{align} 
for $m\in\Z$. Note that this operator can be defined in more general $L^p_h$-spaces, $1\leq p\leq\infty$.
Let us also define the injection $i_h: L^2_{2h}\rightarrow L^2_h$:
\begin{equation}\label{eq:definitionfilter2}
 (i_h f)(x)=\left\lbrace 
 \begin{array}{ll}
f(x) & \mbox{if}\ x\in 2h\Z,\\
0 & \mbox{if}\ x\in h\Z-2h\Z.
\end{array}\right.
\end{equation}

The following simple result allows us to compare the norms of filtered data. The proof is standard, so we omit it.

\begin{lem}\label{thm:uniformh2} For any $f_{2h}:2h\Z \rightarrow\C$ in $L^2_{2h}$ we have that 
\[ \norm{f_{2h}}_{L^2_{2h}}\sim \norm{\Pi_h f_{2h}}_{L^2_h} .\]
\end{lem}
%\begin{proof}
%Note that 
%\begin{align*}
% \norm{f_{2h}}_{\ell^2(2h\Z)}^2  & = \sum_{k\in\Z} |f_{2h}(x_{2k})|^2 \\
% & \leq \sum_{k\in\Z} |f_{2h}(x_{2k})|^2+ \Big |\frac{f_{2h}(x_{2k})+f_{2h}(x_{2k+2})}{2} \Big |^2\\
% & = \norm{\Pi_h f_{2h}}_{\ell^2(h\Z)}^2.
% \end{align*}
%\end{proof}

The following result is based on \cite[Lemma~3.1]{zuazua}, and it states that filtering is equivalent to applying a Fourier multiplier operator. The key idea is that this multiplier vanishes at the critical points of $w(\xi)$ (and nowhere else).

\begin{lem} Let $f_{2h}: 2h\Z \rightarrow \C$ be a function in $L^2_{2h}$. Then for all $\xi\in [-\pi,\pi]$,
\begin{equation}\label{eq:defmultPi}
 \widehat{\Pi_h f_{2h}}(\xi)= 2\,\cos^2\!\left(\frac{\xi}{2}\right) \, \widehat{i_h f_{2h}}(\xi).
\end{equation}
\end{lem}
\begin{proof} Consider $f_{2h}: 2h\Z \rightarrow \C$ of rapid decay. Then we have
\begin{align*}
\widehat{\Pi_h f_{2h}}(\xi) & = \sum_{k\in\Z} f_{2h}(x_{2k}) \, e^{-i2\xi k} + \sum_{k\in\Z} \frac{f_{2h}(x_{2k})+f_{2h}(x_{2k+2})}{2} \, e^{-i\xi (2k+1)}\\
& = \sum_{k\in\Z} f_{2h}(x_{2k}) \, \left(\frac{e^{-i\xi (2k-1)}}{2}+e^{-i2\xi k}+\frac{e^{-i\xi (2k+1)}}{2}\right)\\
& = \sum_{k\in\Z} f_{2h}(x_{2k}) \, e^{-i2k\xi}\, \left(1+\cos\xi\right)= (1+\cos\xi) \, \widehat{i_h f_{2h}}(\xi) = 2\,\cos^2\!\left(\frac{\xi}{2}\right) \, \widehat{i_h f_{2h}}(\xi).
\end{align*}
\end{proof}

As we will see in the next section, although the smoothing effect does not hold in $L^2_h$, it does hold in the subspace $\Pi_h L^2_{2h}\subset L^2_h$.

\section{Lwp of the discrete model}

Based on the approach discussed in the previous section, consider:
\begin{align}\label{eq:inhomogeneous2}
 \left\lbrace \begin{array}{ll}
 i^{\beta}\, \pa_t^{\beta} u_h = (-\Delta_h)^{\frac{\alpha}{2}} u_h + g_h,\quad (t,x)\in [0,T]\times h\Z,\\
 u_h |_{t=0}=\Pi_h f_{2h},
 \end{array} \right.
\end{align}
where $\alpha\in (1,2)$, $\beta \in (\frac{1}{2},1)$, and $f_{2h}$ is the discretization of $f\in H^s(\R)$ as defined in \eqref{eq:defdiscretization}.

We will also take 
\[ g_h (t,x)=\pm \Pi_h\, R_h \left[ |u_h(t,x)|^{p-1} u_h(t,x) \right]\ \mbox{for}\ (t,x)\in [0,T]\times h\Z,\]
where $\Pi_h$ was defined in \eqref{eq:definitionfilter}, and $R_h: L^2_h \rightarrow L^2_{2h}$ is the ``restriction'' operator which takes a function on the lattice $h\Z$ to a function on the lattice $2h\Z$, i.e.
\[ R_h f_h (x)=f_h (x)\ \mbox{for}\ x\in 2h\Z. \]

The proof that \eqref{eq:inhomogeneous2} is locally well-posed, which we stated as \Cref{thm:discretelwp}, is analogous to its continuous version \cite[Theorem~1.2]{mypaper}, except for a few details that must be handled carefully. In order to keep the exposition brief, we devote the remainder of this section to explain what those differences are, exemplified by the most important linear estimates needed in the proof of \Cref{thm:discretelwp}.

Let us write the initial value problem associated to \eqref{eq:inhomogeneous2} as:
\begin{align*}
u_h (t) & = L_{t,h} \Pi_h f_{2h} \pm i^{-\beta}\,\int_0^t N_{t-t',h}\left[ \Pi_h R_h \left( |u_h (t')|^{p-1} R_h u_h (t')\right)\right] \, dt'\\
 & = L_{t,h} \Pi_h f_{2h} \pm i^{-\beta}\,\int_0^t N_{t-t',h} g_h(t') \, dt',
\end{align*}
where we use the following notation:
\[
L_{t,h} f_h =\left(  E_{\beta}(i^{-\beta} t^{\beta} h^{-\alpha} w(\cdot) ) \widehat{f_h}\right)^{\vee},\qquad N_{t,h} g_h =\left(  t^{\beta-1}\,E_{\beta,\beta}(i^{-\beta} t^{\beta} h^{-\alpha} w(\cdot) ) \widehat{g_h}\right)^{\vee}.
\]

The following theorem is the discrete analog to the results in \Cref{thm:main_cont_est}. We omit the proof since it is just an application of a uniform bound on the Fourier multiplier associated to the operators $L_{t,h}$ and $N_{t,h}$, which follows from the asymptotics in \eqref{eq:MLasymp} and \eqref{eq:genMLasymp}.

\begin{thm}\label{thm:massestimates}
Let $\alpha>0$, $\beta\in (0,1)$ and $\sigma=\frac{\alpha}{\beta}$. Then we have that:
\begin{align*}
 \norm{L_{t,h} f_h}_{L^{\infty}_T L^2_h} & \lesssim \norm{f_h}_{L^2_h}.\\
 \norm{\int_0^t N_{t-t',h} g_h(t')\, dt'}_{L^{\infty}_T L^2_h} & \lesssim \max\{ T^{1/2}, T^{\beta-\frac{1}{2}}\}\,\norm{\langle h^{-1}\nabla\rangle^{\sigma-\alpha}\,g_h}_{L^2_T L^2_h}.
 \end{align*}
\end{thm}
\begin{rk} Let us highlight the loss of $\sigma-\alpha> 0$ derivatives, due to the Fourier multiplier operator $N_{t,h}$ acting on the nonlinear term, see \eqref{eq:genMLasymp}.
\end{rk}

Now let us explain how to prove the smoothing effect, which is now possible thanks to the operator $\Pi_h$.

\begin{prop}[Smoothing effect]\label{thm:discretesmoothing2} Consider $f_{2h}: 2h\Z \rightarrow \C$ in $L^2_{2h}$ and let 
\[\phi_h(\xi):=h^{-\sigma} w(\xi)^{1/\beta}.\]
Then we have that
\[ \norm{ |h^{-1}\,\nabla|^{\frac{\sigma-1}{2}} \, e^{-it \phi_h(\nabla)} \Pi_h f_{2h}}_{L^{\infty}_h L^2_t} \lesssim \norm{\Pi_h f_{2h}}_{L^2_h}.\]
\end{prop}
\begin{proof} 
Let 
\[ W(t) g (x)=\int_{\Omega} e^{-it \phi(\xi)}\, \widehat{g}(\xi)\, e^{ix\xi}\, d\xi.\]
By \cite[Theorem~4.1]{KPV2}, we have that 
\[ \sup_x \norm{W(t) g}^2_{L^2_t}\lesssim \int_{\Omega} \frac{|\widehat{g}(\xi)|^2}{|\phi'(\xi)|}\, d\xi\]
as long as $\phi'\neq 0$ in the open set $\Omega\subset\R$. The result remains true when $\phi'$ has zeroes as long as the right-hand side is finite.

In our case, we set $\Omega=(-\pi,\pi)$ and
\[ \phi_h(\xi) := h^{-\sigma} w(\xi)^{\frac{1}{\beta}},\]
so that, for $x_m=mh$,
\[W(t) \Pi_h f_{2h}(x_m)=\int_{-\pi}^{\pi}\, e^{-it \phi_h (\xi)+im\xi}\, \widehat{\Pi_h f_{2h}}(\xi)\, d\xi.\]

Then \cite[Theorem~4.1]{KPV2}, \Cref{thm:aboutw} and \eqref{eq:defmultPi} yield
\begin{align*}
 \norm{ |\nabla|^{\frac{\sigma-1}{2}} \, e^{-it \phi_h(\nabla)} \Pi_h f_{2h}}^2_{L^{\infty}_h L^2_t} & \lesssim \int_{-\pi}^{\pi} |\widehat{\Pi_h f_{2h}}(\xi)|^2\,\frac{|\xi|^{\sigma-1}}{|\phi_h'(\xi)|}\, d\xi\\
& \lesssim 8\beta h^{\sigma}\, \int_{-\pi}^{\pi} |\widehat{i_h f_{2h}}(\xi)|^{2}\,\frac{\cos^4\!\left(\frac{\xi}{2}\right)\,|\xi|^{\sigma-1}}{ |w'(\xi)|\, |w(\xi)|^{\frac{1}{\beta}-1}}\, d\xi\\
&  \lesssim h^{\sigma}\, \int_{0}^{\pi} (|\widehat{i_h f_{2h}}(\xi)|^{2}+|\widehat{i_h f_{2h}}(-\xi)|^{2})\,\frac{(\pi-\xi)^{4}\,|\xi|^{\sigma-1}}{ |\xi|^{\alpha-1}\, (\pi-\xi)\, |\xi|^{\sigma-\alpha}}\, d\xi\\
&  \lesssim h^{\sigma}\, \int_{-\pi}^{\pi} |\widehat{i_h f_{2h}}(\xi)|^{2}\, d\xi=h^{\sigma-1}\, \norm{i_h f_{2h}}_{L^2_{h}}^2 \lesssim h^{\sigma-1}\, \norm{\Pi_h\,f_{2h}}_{L^2_h}^2 .
\end{align*}
\end{proof}

The rest of the estimates that involve the smoothing effect admit a similar proof, so we omit it. We summarize them in the following theorem, which is the discrete version of \Cref{thm:main_cont_est}.

\begin{thm}\label{thm:smoothingestimates}
Let $\alpha>1$, $\beta\in (\frac{1}{2},1)$ and $\sigma=\frac{\alpha}{\beta}$.
Then we have that
\begin{align*}
 \norm{ \langle h^{-1}\,\nabla\rangle^{\frac{\sigma-1}{2}-} \, L_{t,h} \Pi_h f_{2h}}_{L^{\infty}_h L^2_T} & \lesssim  \langle T\rangle^{\frac{1}{2}+}\, \norm{\Pi_h f_{2h}}_{L^2_h},\\
  \norm{ \langle h^{-1}\,\nabla\rangle^{\frac{\sigma-1}{2}-} \int_0^t N_{t-t',h} \Pi_h g_{2h}(t')\, dt'}_{L^{\infty}_h L^2_T} & \lesssim T^{\frac{1}{2}}\, \langle T\rangle^{\frac{1}{2}+}\, \norm{\Pi_h g_{2h}}_{L^2_T H^{\sigma-\alpha}_h}.
 \end{align*}
\end{thm}

We now explain how to prove our maximal function estimates. Not only does this estimate require handling critical points of $\phi_h(\xi)$, but also zeroes of 
the second derivative $\phi_h''(\xi)$. This was not a problem when we had $|\xi|^{\sigma}$, but now it will be given the following result.

\begin{lem}\label{thm:phasefunction} For $\alpha\in (1,2)$ and $\beta \in (\frac{1}{2},1)$, define 
\[\phi_h(\xi):=h^{-\sigma} w(\xi)^{1/\beta}.\]
Then $\phi_h''(\xi)$ has a unique zero $\xi_1\in [0,\pi]$. 
Moreover, $\xi_1>\xi_0$, where $\xi_0\in (0,\pi)$ is the unique zero of $w''$ given by \Cref{thm:aboutw}.
\end{lem}
\begin{proof}
Without loss of generality, we may assume that $h=1$. Then a zero, $\xi_1$, of $\phi''_1$ must satisfy:
\begin{equation}\label{eq:second_derivative}
 \left(\frac{1}{\beta}-1\right) w(\xi_1)^{\frac{1}{\beta}-2} w'(\xi_1)^2+  w(\xi_1)^{\frac{1}{\beta}-1}\, w''(\xi_1)=0.
 \end{equation}
By \Cref{thm:aboutw}, the first summand in \eqref{eq:second_derivative} is positive, and $w''$ is monotone decreasing with a unique zero $\xi_0$. 
Since $w$ is positive everywhere, there can be no zero of $\phi''_1$ in the region where $w''$ is positive. 

Note also that $\phi''_1 (\pi)=  w(\pi)^{\frac{1}{\beta}-1}\, w''(\pi)<0$ thus there must be at least one zero $\xi_1$, which must lie in the interval $(\xi_0,\pi)$, the region where $w''<0$. As $\xi\in (\xi_0,\pi)$ grows, $w'(\xi)$ decreases (because $w''<0$ there) and $w(\xi)$ increases. Consequently, the first summand in \eqref{eq:second_derivative},
\[ \frac{w'(\xi)^2}{w(\xi)^{2-\frac{1}{\beta}}}, \]
must decrease as $\xi\in (\xi_0,\pi)$ grows (note that the exponent $2-\frac{1}{\beta}$ is positive).

The second summand in \eqref{eq:second_derivative}, $w(\xi)^{\frac{1}{\beta}-1}\, w''(\xi)$, becomes more negative as $\xi\in (\xi_0,\pi)$ grows, because $w$ is increasing and $w''$ is negative and monotone decreasing. Thus the sum of the first and second summands must become more negative as $\xi\in (\xi_0,\pi)$ grows. Consequently, there can only be one zero.
\end{proof}

\begin{prop}[Maximal function]\label{thm:discretemaximal}
Consider $f_{h}: h\Z \rightarrow \C$ in $L^2_h$ and let 
\[\phi_h(\xi):=h^{-\sigma} w(\xi)^{1/\beta}.\]
Then for $s=\frac{1}{2}-\frac{1}{p}$, $p\in [4,\infty)$ and $\sigma>1$, we have
\begin{equation}\label{eq:KdVtypebound}
 \norm{ e^{-it \phi_h(\nabla)} f_h}_{L^{p}_h L^{\infty}_t} \lesssim \norm{|h^{-1}\,\nabla|^{s} f_h}_{L^2_h}.
 \end{equation}
\end{prop}
\begin{proof}
{\bf Step 1: Uniform decay.} We will use the same strategy of proof as in Lemma 3.29 in \cite{KPV}, with a few key changes.

For $m\in\Z$, consider 
\[
 I(t,h,m)  = \int_{-\pi}^{\pi} e^{-it \phi_h(\xi)+i m\xi}\, \frac{h^{s}}{|\xi|^{s}}\, d\xi = h\, \int_{-	\frac{\pi}{h}}^{\frac{\pi}{h}} e^{-it \phi_h(h\xi)+i mh\xi}\, \frac{1}{|\xi|^{s}}\, d\xi.
\]
For $\frac{1}{2}\leq s<1$, we wish to establish the estimate
\begin{equation}\label{eq:KdVtypedecay}
|I(t,h,m)|\lesssim h^s \, |m|^{s-1} ,
\end{equation}
where the implicit constant is independent of $t$ and $h$.

We subdivide the region of integration into:
\begin{align*}
\Omega_1 & := \lbrace \xi\in \left[-\frac{\pi}{h},\frac{\pi}{h}\right] \mid |\xi|\leq |mh|^{-1},\ \mbox{or}\ |\xi\pm h^{-1}\xi_j|\lesssim |mh|^{-1}\ \mbox{for}\ j=1,2\rbrace,\\
\Omega_2 & := \{ \xi\in \left[-\frac{\pi}{h},\frac{\pi}{h}\right]-\Omega_1\mid | mh-th\,\phi_h'(h\xi)|\leq \frac{\sigma-1}{2\sigma}\, |mh|\},\\
\Omega_3 & :=  \left[-\frac{\pi}{h},\frac{\pi}{h}\right] - \Omega_1 \cup \Omega_2,
\end{align*}
where $\pm\xi_1$ are the zeroes of $\phi_h''$ and $\pm\xi_2$ those of $\phi_h'$.

For $j=1,2,3$, we write
\[ I_j (t,h,m):= h\, \int_{\Omega_j} e^{-it \phi_h(h\xi)+i mh\xi}\, \frac{1}{|\xi|^{s}}\, d\xi.\]

In the case of $\Omega_1$, the decay given by \eqref{eq:KdVtypedecay} follows directly from integrating the absolute value of the integrand.

On $\Omega_2$, we can use \Cref{thm:phasefunction} to argue that $|h^2\,\phi_h''(h\xi)|\sim |\xi|^{\sigma-2}$. This is its behavior near zero, but we have removed a neighborhood of its only zero, so we may use continuity elsewhere to extend the bound to the whole set $\Omega_2$. 

By the Van der Corput lemma, we have:
\[
 |I_2(t,h,m)|  \lesssim h\, \sup_{\xi\in\Omega_2}|\xi|^{-s}\, (t \, h^2 \, \phi_h''(h\xi))^{-1/2} \lesssim h\, \sup_{\xi\in\Omega_2}|\xi|^{-s}\,(t \, |\xi|^{\sigma-2})^{-1/2}
\]
On $\Omega_2$ we also have that $|\xi|^{\sigma-1}\sim |h\,\phi_h'(h\xi)|\sim \frac{|mh|}{t}$, since we removed $\xi_2$, the only point where $\phi_h'$ vanishes. Additionally, $|\xi|>|mh|^{-1}$ in this region. All in all,
\begin{align*}
 |I_2(t,h,m)| & \lesssim h\, \sup_{\xi\in\Omega_2}|\xi|^{-s}\,(t \, |\xi|^{\sigma-2})^{-1/2} \lesssim h\, \sup_{\xi\in\Omega_2}|\xi|^{-s}\,(|mh| \, |\xi|^{-1})^{-1/2}\\
 & \lesssim h\, \sup_{\xi\in\Omega_2}|\xi|^{\frac{1}{2}-s}\,|mh|^{-1/2} \lesssim h\, |mh|^{s-\frac{1}{2}}\,|mh|^{-1/2}=h^{s}\, |m|^{s-1}.
\end{align*}

To prove the desired decay over $\Omega_3$, we integrate by parts, having excluded all problematic points on $\Omega_1$ and all points where the phase vanishes on $\Omega_2$.

{\bf Step 2: $TT^{\ast}$ argument.} Now we explain how to obtain \eqref{eq:KdVtypebound} from \eqref{eq:KdVtypedecay}.

By duality, \eqref{eq:KdVtypebound} is equivalent to 
\[ \norm{ \int_{\R} |h^{-1}\nabla|^{-s}\, e^{-it \phi_h(\nabla)} g_h (t,\cdot) \, dt}_{L^{2}_h} \lesssim \norm{g_h}_{L^{p'}_h L^1_t},\]
where $p'$ is the dual exponent of $p$. However, we also have that 
\begin{multline*}
\norm{ \int_{\R}|h^{-1}\nabla|^{-s}\, e^{-it \phi_h(\nabla)} g_h(t,\cdot) \, dt}_{L^{2}_h}^2 = \\
 h\, \sum_{m\in\Z} \int_{\R} g_h(t,mh) \left( \int_{\R} |h^{-1}\nabla|^{-2s}\, e^{-i(t-\tau) \phi_h(\nabla)} \overline{g_h(\tau,mh)}\, d\tau\right) \, dt.
\end{multline*}
Therefore, \eqref{eq:KdVtypebound} is also equivalent to
\[ \norm{ \int_{\R} |h^{-1}\nabla|^{-2s}\, e^{-i(t-\tau) \phi_h(\nabla)} g_h(\tau,\cdot)\, d\tau}_{L^{p}_h L^{\infty}_t} \lesssim \norm{g_h}_{L^{p'}_h L^{1}_t}.\]

In order to prove this last bound, we use \eqref{eq:KdVtypedecay} and the Hardy-Littlewood-Polya inequality, which is a discrete version of the well-known Hardy-Littlewood-Sobolev inequality (see \cite{SteinWainger}):
\begin{align*}
\norm{ \int_{\R} |h^{-1}\nabla|^{-2s}\, e^{-i(t-\tau) \phi_h(\nabla)} g_h(\tau,mh)\, d\tau}_{L^{p}_h L^{\infty}_t} & = 
\norm{ |h^{-1}\nabla|^{-2s}\, e^{-it \phi_h(\nabla)} \ast_{t,m} g_h}_{L^{p}_h L^{\infty}_t}\\
& \leq \norm{ \Big| |h^{-1}\nabla|^{-2s}\, e^{-it \phi_h(\nabla)}\Big | \ast_{t,m} |g_h|}_{L^{p}_h L^{\infty}_t}\\
& \lesssim \norm{ h^{2s}\,|m|^{2s-1}\ast_{t,m} |g_h|}_{L^{p}_h L^{\infty}_t}\\
& \leq  h^{2s}\,\norm{|m|^{2s-1}\ast_{m} \norm{g_h}_{L^1_t}}_{L^{p}_h}\\
& \lesssim  \norm{\norm{g}_{L^1_t}}_{L^{p'}_h}.
\end{align*}
Note that the factor $h^{2s+\frac{1}{p}}$ from the $L^p_h$-norm is equal to the factor $h^{\frac{1}{p'}}$ present in the $L^{p'}_h$-norm.
\end{proof}

The rest of the maximal function estimates required to prove \Cref{thm:discretelwp} are analogous to those found in the continuous setting, see \Cref{thm:main_cont_est}, as long as one handles zeroes of $\phi_h'$ and $\phi_h''$ as explained in these two examples. We summarize them below.

\begin{thm}\label{thm:maximalestimates} 
Let $\alpha>1$, $\beta\in (\frac{1}{2},1)$ and $\sigma=\frac{\alpha}{\beta}$. For $p\geq 4$ and $s=\frac{1}{2}-\frac{1}{p}$ we have:
\begin{align*}
\norm{ L_{t,h} f_h}_{L^{p}_h L^{\infty}_T} & \lesssim \norm{\langle h^{-1}\,\nabla\rangle^{s} f_h}_{L^2_h},\\
\norm{\int_0^t N_{t-t',h} \, g_h(t')\, dt'}_{L^{p}_h L^{\infty}_T} & \lesssim \max\{ T^{\frac{1}{2}}, T^{\beta-\frac{1}{2}}\}\, \norm{\langle h^{-1}\,\nabla\rangle^{s+\sigma-\alpha} g_h}_{L^2_T L^2_h}.
\end{align*}
\end{thm}

Finally, we interpolate between the maximal function estimates and the smoothing effect. Luckily, the fact that we are working with discrete spaces makes this step simpler, since we do not need to work with functions with bounded mean oscillation thanks to Theorem 5.6.3 in \cite{interpolation}, which we record below:
\begin{lem} Let $A$ be a Banach space, $s\in\R$ and $q\in (0,\infty]$. Let $\ell_q^s (A)$ be the space of sequences $(a_n)_{n\in\N} \subset A$ (the set of subindices may also be $\Z$) such that 
\[ \left(\sum_{n\in N} \left( 2^{ns}\, \norm{a_n}_A\right)^q \right)^{1/q} <\infty.\]
Then for $\theta\in (0,1)$, $s_1,s_2\in\R$ and Banach spaces $A_0, A_1$, we have that
\[ \left[\ell_{q_0}^{s_0} (A_0) , \ell_{q_1}^{s_1} (A_1) \right]_{\theta} = \ell_q^s \left( [A_0, A_1]_{\theta}\right),\]
where
\[ s= (1-\theta)\, s_0 + \theta\, s_1,\quad \frac{1}{q}=\frac{1-\theta}{q_0} + \frac{\theta}{q_1}.\]
\end{lem}
\begin{rk}
For our purposes, $(q_0,q_1)=(\infty,p)$, $A_0 =L^2_T$ and $A_1=L^{\infty}_T$. Note that Theorem 5.1.1 in \cite{interpolation} guarantees that $[A_0, A_1]_{\theta}=L^{p_{\theta}}$ for $\frac{1}{p_{\theta}}= \frac{1-\theta}{2}$.
\end{rk}
\begin{rk} Let us highlight the difference between this result and Theorem 5.1.2 in \cite{interpolation}, which is the continuous case (with $L^q$ instead of $\ell_q$). In the latter, the endpoint $q_0=\infty$ is not included, which is why we need $\mbox{BMO}_x$.
\end{rk}

This result together with the Stein (complex) interpolation theorem yields the following 

\begin{thm}\label{c2:thm:interpolation}
Let $\alpha>1$, $\beta\in (\frac{1}{2},1)$, $\sigma=\frac{\alpha}{\beta}$, $\gamma=\frac{\sigma-1}{2}$ and $\tilde{\gamma}=\alpha-\frac{\sigma+1}{2}$. For $p\geq 3$, $s=\frac{1}{2}-\frac{1}{2(p-1)}$, $0\leq \gamma'<\gamma$ and $0\leq \tilde{\gamma}'<\tilde{\gamma}$ we have:
\begin{align*}
\norm{ L_{t,h} \Pi_h f_{2h}}_{L^{4(p-1)}_h L^{4}_T} & \lesssim \langle T\rangle^{\frac{1}{2}+}\, \norm{\langle h^{-1}\,\nabla\rangle^{\frac{s-\gamma'}{2}} \Pi_h f_{2h}}_{L^2_h},\\
\norm{\int_0^t N_{t-t',h} \, \Pi_h g_{2h}(t')\, dt'}_{L^{4(p-1)}_h L^{4}_T} & \lesssim \max\{T^{\frac{1}{2}}, T^{\frac{\beta}{2}}\}\, \langle T\rangle^{\frac{1}{2}+}\, \norm{\langle h^{-1}\,\nabla\rangle^{\frac{\sigma-\alpha +s-\tilde{\gamma}'}{2}}  \Pi_h g_{2h}}_{L^2_T L^2_h}.
\end{align*}
\end{thm}

These are the linear estimates needed to prove \Cref{thm:discretelwp}. The argument, based on the Banach fixed point theorem, is analogous to that in \cite{mypaper} and \cite[Section~2.4]{mythesis}.

\section{Continuum limit}

\subsection{General strategy} By \Cref{thm:discretelwp}, for all small $h>0$ we have a solution $u_h:[0,T]\times h\Z \rightarrow \C$ to the discrete problem given by \eqref{eq:inhomogeneous2}, where $T$ only depends on the norm of the initial data $\norm{f}_{H^s_x}$. By \Cref{thm:contLWP}, we also have a solution $u:[0,T]\times\R\rightarrow \C$ to the continuous problem, where we may assume that $T$ is the smallest of both such intervals of existence.

Consider the linear interpolation of $u_h$, which we denote by $p_h u_h:[0,T]\times\R\rightarrow \C$. For $x_m=hm\in h\Z$ and $x\in [x_m, x_{m+1})$ we define:
\begin{equation}\label{eq:deflinearinterpolation}
p_h u_h (t,x)= u_h (t,x_m) + D_h^{+}u_h (t,x_m)\cdot (x-x_m),
\end{equation}
where $D_h^{+}$ is the forward difference defined in \eqref{eq:defForwardDifference}. 

Our goal is to show that the interpolation of the discrete solution, $p_h u_h$, converges to the continuous solution $u$ in some way. In \cite{KLS}, Kirkpatrick, Lenzmann and Staffilani prove that $\{p_h u_h\}_{h>0}$ and $\{\pa_t p_h u_h\}_{h>0}$ are uniformly bounded in $L^{\infty}_T H^{\alpha/2}_x$ and $L^{\infty}_T H^{-\alpha/2}_x$, respectively. Then the Banach-Alaoglu theorem allows them to extract a weak-$\ast$ convergent subsequence, whose limit is shown to be $u$, the solution of the continuous IVP.

This result was improved to strong convergence in \cite{Hong}. Their approach is based on studying the difference $\norm{u(t)-p_h u_h(t)}_{L^2_x}$ by using their respective IVPs and a careful estimation of the error terms, together with the Gronwall inequality. Although more technical than \cite{KLS}, this approach seems quite robust 
to prove continuum limits of more general discrete problems, and so it will be our choice.

The goal in this section is to generalize these ideas to work in more general spaces based on $L^p_x L^q_T$, which is what our equation requires. Instead of the Gronwall inequality, we will use the method of continuity to obtain strong convergence.

In particular, consider the continuous and discrete IVPs:
\begin{align*}
u(t) & = L_t f \pm i^{-\beta}\,\int_0^t N_{t-t'}\left( |u (t')|^{p-1} u (t')\right) \, dt',\\
p_h u_h (t) & = p_h (L_{t,h} \Pi_h f_{2h}) \pm i^{-\beta}\, p_h\,\int_0^t N_{t-t',h}\left[ \Pi_h R_h \left( |u_h (t')|^{p-1} R_h u_h (t')\right)\right] \, dt',
\end{align*}
where we use the following notation:
\begin{align}\label{eq:defcompareIVP}
L_t f & =\left(  E_{\beta}(i^{-\beta} t^{\beta} |\cdot |^{\alpha}) \widehat{f}\right)^{\vee},\\
L_{t,h} f_h & =\left(  E_{\beta}(i^{-\beta} t^{\beta} h^{-\alpha} w(\cdot) ) \widehat{f_h}\right)^{\vee},\\
N_t g & =\left(  t^{\beta-1}\, E_{\beta,\beta}(i^{-\beta} t^{\beta} |\cdot |^{\alpha}) \widehat{g}\right)^{\vee},\\
N_{t,h} g_h & =\left(  t^{\beta-1}\,E_{\beta,\beta}(i^{-\beta} t^{\beta} h^{-\alpha} w(\cdot) ) \widehat{g_h}\right)^{\vee}.
\end{align}
Note that we will use the same notation for the continuous Fourier transform and the discrete Fourier transform as long as there is no risk of misinterpretation.

The main idea is to consider the difference $u-p_h u_h$ in the space where we we have local well-posedness for the continuous equation. This requires studying the operators $L_t - p_h L_{t,h}$ and $N_t - p_h N_{t,h}$ in the appropriate norms to exploit cancellation and obtain some terms that are $o_h(1)$ as $h\rightarrow 0$, as well as some other terms that can be controlled in terms of the initial data and $u-p_h u_h$.

In the rest of this section, we first introduce some useful results about linear interpolation $p_h$. Then we study the operators $L_t - p_h L_{t,h}$ and $N_t - p_h N_{t,h}$ in the appropriate norms. Finally, we give the proof of the continuum limit.

\subsection{Interpolation}
We summarize some results concerning the operator $p_h$. Most of these results are well-known so we will omit the proofs.

\begin{lem}[Lemma 3.7 in \cite{KLS}] For $0\leq s \leq 1$ we have
\[ \norm{p_h f_h}_{H^s_x}\lesssim \norm{f_h}_{H^s_h}.\]
\end{lem}

The proof is based on complex interpolation between $L^2_h$ and $H^1_h$, where the equivalent norm given by forward differences is useful.

The following result can be proved by a direct computation of the Fourier transform.

\begin{lem}[Lemma 5.5 in \cite{Hong}] Let $f_h\in L^2_h$. Then 
\begin{equation}\label{eq:FTinterpolation}
\widehat{p_h f_h}(\xi) = P_h(\xi) \, \widehat{f_h}(h\xi),
\end{equation}
where 
\[ P_h(\xi):= \int_0^{h} e^{-ix\xi}\, dx+  \frac{e^{ih\xi}-1}{h}\, \int_{0}^{h} x \, e^{-ix\xi} \, dx.
\]
\end{lem}
\begin{rk} Note that each side of \eqref{eq:FTinterpolation} represents a different type of Fourier transform (continuous and discrete), but the equality holds for all $\xi\in\R$ by taking the periodization of the right-hand side.
\end{rk}

We now present the relationship between discretization and interpolation, which is based on \cite[Proposition~5.3]{Hong}. See also \cite[Chapter~6]{lady}.

\begin{lem}\label{thm:gainh} Let $f\in H^{s_2}_x(\R)$ and let $f_h$ be its discretization according to \eqref{eq:defdiscretization}.   For $0\leq s_1\leq s_2 \leq 1$ we have
\[ \norm{ p_h f_h - f}_{H^{s_1}_x} \lesssim h^{s_2-s_1} \, \norm{f}_{H^{s_2}_x},\]
where the implicit constant does not depend on $h$.
\end{lem}

The results summarized so far will be necessary tools when studying $L_t - p_h L_{t,h}$. We now present some other results that will allow us to estimate the nonlinearity $N_t - p_h N_{t,h}$. The first one is a generalization of \cite[Proposition~5.8]{Hong}. 

\begin{lem}\label{thm:honglemma} Let $p\geq 3$ be an odd integer. For $0\leq s_1\leq s_2 \leq 1$ we have
\[ \norm{ |p_h u_h|^{p-1} p_h u_h - p_h ( |u_h|^{p-1} u_h )}_{H^{s_1}_x} \lesssim h^{s_2-s_1} \, \norm{u_h}^{p-1}_{L^{\infty}_h} \, \norm{u_h}_{H^{s_2}_h} ,\]
where the implicit constant does not depend on $h$.
\end{lem}
\begin{proof}
Fix $s_1\geq 0$. The proof is an application of complex interpolation between the cases $(s_1,s_2)=(0,0), (0,1)$ and $(1,1)$. The case $(0,0)$ is trivial and follows from the H\"older inequality.

{\bf Step 1.} Consider the case $(s_1,s_2)=(1,1)$. It suffices to show that:
\[ \norm{ |p_h u_h|^{p-1} p_h u_h - p_h ( |u_h|^{p-1} u_h )}_{\dot{H}^1_x} \lesssim \norm{u_h}^{p-1}_{L^{\infty}_h} \, \norm{u_h}_{\dot{H}^1_h} .\]

Suppose that $x\in [x_m, x_{m+1})$. We write:
\begin{align}\label{eq:pdifference}
|p_h u_h(x)|^{p-1} & p_h u_h(x) - p_h ( |u_h|^{p-1} u_h )(x)  \nonumber\\
 =&\  |p_h u_h(x)|^{p-1} p_h u_h(x)  - |u_h(x_m)|^{p-1} u_h(x_m)\nonumber\\
&  - \frac{|u_h(x_{m+1})|^{p-1} u_h(x_{m+1}) - |u_h(x_m)|^{p-1} u_h(x_m) }{h} \cdot (x-x_m).
\end{align}

Recall that $p\geq 3$ is an odd integer. If $p=3$, we have that 
\[
|a|^2 a - |b|^2 b = (|a|^2 + b \bar{a})\, (a-b) + b^2 (\bar{a}-\bar{b}).
\]
The same ideas show that, for any odd $p\geq 3$, one can write
\begin{equation}\label{eq:ap_bp}
  |a|^{p-1} a - |b|^{p-1} b = P_1 (a,b,\bar{a},\bar{b})\, (a-b)+ P_2 (a,b,\bar{a},\bar{b})\, (\bar{a}-\bar{b})
  \end{equation}
where $P_1$ and $P_2$ are homogeneous polynomials of degree $p-1$.

We use this equality on the first term in \eqref{eq:pdifference}. Since $p-1$ is even, we can consider the cases where the derivative hits each of the terms.
We use the H\"older inequality to put the terms with a derivative in $\dot{H}^1_x$ and all the others in $L^{\infty}_x$. In all cases we obtain:
\[
\norm{|p_h u_h(x)|^{p-1} p_h u_h(x)  - |u_h(x_m)|^{p-1} u_h(x_m)}_{\dot{H}^1_x ([x_m,x_{m+1}))}
\lesssim h^{1/2}\, \norm{u_h}_{L^{\infty}_h}^{p-1} \, |D_h^{+} u_h (x_m)|.
\]
As before, taking squares and summing in $m\in\Z$ yields the desired inequality. The second term in \eqref{eq:pdifference} may be treated analogously.

{\bf Step 2.} Finally, we prove the estimate for $(s_1,s_2)=(0,1)$. Consider the first term in \eqref{eq:pdifference}. When $x\in [x_m,x_{m+1})$, we can estimate this by:
\begin{multline*}
\Big | |p_h u_h(x)|^{p-1} p_h u_h(x) - |u_h(x_m)|^{p-1} u_h(x_m)\Big| \\
 \lesssim \left( |p_h u_h(x)|^{p-1} +|u_h(x_m)|^{p-1} \right) \, |p_h u_h(x) - u_h(x_m)| \\
 =  \left( |p_h u_h(x)|^{p-1} +|u_h(x_m)|^{p-1} \right) \, |D_h^{+} u_h (x_m)\,(x-x_m)|.
\end{multline*}
We square this and integrate over $x\in [x_m,x_{m+1})$.
\[
\norm{ |p_h u_h(x)|^{p-1} p_h u_h(x) - |u_h(x_m)|^{p-1} u_h(x_m)}_{L^2_x ([x_m,x_{m+1}))}^2 
 \lesssim  h^3\, \norm{u_h}^{2(p-1)}_{L^{\infty}_h} \, |D_h^{+} u_h (x_m)|^2 .
\]
We finally sum in $m\in\Z$ and obtain:
\[ \norm{ |p_h u_h(x)|^{p-1} p_h u_h(x) - |u_h(x_m)|^{p-1} u_h(x_m)}_{L^2_x (\R)}^2
 \lesssim  h^2\, \norm{u_h}^{2(p-1)}_{L^{\infty}_h} \, \norm{u_h}_{H^1_x}^2.\]
The second term in \eqref{eq:pdifference} admits a similar argument.
\end{proof}

\subsection{Comparison estimates}

We start this section with a lemma that will allow us to estimate the increment of the Mittag-Leffler functions at infinity.

\begin{lem} Let $\beta\in (0,1)$ and $z_1,z_2\in\C$ such that $|z_i|\geq M\geq 1$ and $\arg z_i=-\frac{\pi}{2}\,\beta$ for $i=1,2$. Then
\begin{equation}\label{eq:compareML}
 E_{\beta}(z_1)-E_{\beta}(z_2)=\frac{1}{\beta} e^{-i|z_1|^{1/\beta}} - \frac{1}{\beta} e^{-i|z_2|^{1/\beta}} + \O \left( \frac{|z_1-z_2|}{|z_1|\, |z_2|}\right).
\end{equation}
\end{lem}
\begin{proof}
The Mittag-Leffler function admits the following integral representation:
\[ E_{\beta}(z)=\frac{1}{2\pi i} \int_C \frac{t^{\beta-1} \, e^t}{t^{\beta}-z}\, dt,\]
where we choose the branch of the logarithm where $-\pi < \arg z\leq \pi$ and $C$ is the Hankel contour: it starts and ends at $-\infty$ and it circles around the disk $|t|>|z|^{1/\beta}$ counterclockwise (see \cite{bate}).

There is only one pole of the integrand in this branch, which is given by $t_p=-i|z|^{1/\beta}$. We then deform the contour $C$ to a new contour $\widetilde{C}$ that starts and ends at $-\infty$ and it circles around the disk $|t|\geq\varepsilon$ counterclockwise. In doing so, the pole falls outside of $\widetilde{C}$ so we pick up its residue:
\[
 \lim_{t\rightarrow t_p} (t-t_p) \, \frac{t^{\beta-1} \, e^t}{t^{\beta}-z}  = t_p^{\beta-1}\, e^{t_p}\, \lim_{t\rightarrow t_p} \frac{t-t_p}{t^{\beta}-t_p^{\beta}} = t_p^{\beta-1}\, e^{t_p}\, \frac{1}{\beta} \, t_p^{1-\beta} =  \frac{1}{\beta}\,e^{t_p}= \frac{1}{\beta} e^{-i|z|^{1/\beta}}.
\]
All in all, we have
\[ E_{\beta}(z)= \frac{1}{\beta} e^{-i|z|^{1/\beta}} + \frac{1}{2\pi i}\, \int_{\widetilde{C}} \frac{t^{\beta-1} \, e^t}{t^{\beta}-z}\, dt.\]
Suppose we have $z_1,z_2\in \C$ with $|z_i|\geq M\geq 1$ and $\arg z_i=-\frac{\pi}{2}\,\beta$ for $i=1,2$. Then we can write
\[ \Big |  1-\frac{t}{z_i} \Big | \geq 1-\frac{\varepsilon}{M}=c>0,\quad \mbox{for all}\ t\in \widetilde{C}.\]
Consequently,
\[
 E_{\beta}(z_1)-E_{\beta}(z_2)= \frac{1}{\beta} e^{-i|z_1|^{1/\beta}} - \frac{1}{\beta} e^{-i|z_2|^{1/\beta}} + \frac{1}{2\pi i}\, \int_{\widetilde{C}} t^{\beta-1} \, e^t\,\left( \frac{1}{t^{\beta}-z_1}-\frac{1}{t^{\beta}-z_2}\right)\, dt.
\]
It is now easy to show that 
\begin{align*}
 \Big | \int_{\widetilde{C}} t^{\beta-1} \, e^t\,\left( \frac{1}{t^{\beta}-z_1}-\frac{1}{t^{\beta}-z_2}\right)\, dt \Big | & = \Big | \int_{\widetilde{C}} t^{\beta-1} \, e^t\, \frac{z_1-z_2}{(t^{\beta}-z_1)(t^{\beta}-z_2)}\, dt\Big |\\
 & \leq \frac{|z_1-z_2|}{c^2\, |z_1|\, |z_2|} \, \int_{\widetilde{C}} |t^{\beta-1} \, e^t|\, dt.
 \end{align*}
 The last integral converges thanks to the exponential.
\end{proof}
\begin{rk} Note that we could also prove this result by writing $E_{\beta}(z_1)-E_{\beta}(z_2)$ in terms of the derivative of $E_{\beta}$ and proving a uniform bound for this derivative. We chose the more general proof above to give the main idea about how asymptotics (and uniform bounds) for the Mittag-Leffler function are obtained.
\end{rk}

The same idea can be applied to obtain error bounds for the increment of the generalized Mittag-Leffler function.

\begin{lem} Let $\beta\in (0,1)$ and $z_1,z_2\in\C$ such that $|z_i|\geq M\geq 1$ and $\arg z_i=-\frac{\pi}{2}\,\beta$ for $i=1,2$. Then
\begin{equation}\label{eq:comparegenML}
 E_{\beta,\beta}(z_1)-E_{\beta,\beta}(z_2)=\frac{i^{\beta-1}}{\beta}\, |z_1|^{\frac{1}{\beta}-1} \, e^{-i|z_1|^{1/\beta}} - \frac{i^{\beta-1}}{\beta}\, |z_2|^{\frac{1}{\beta}-1}\, e^{-i|z_2|^{1/\beta}} + \O \left( \frac{|z_1-z_2|}{|z_1|\, |z_2|}\right).
\end{equation}
\end{lem}
\begin{proof}
In this case one uses the integral representation 
\[ E_{\beta,\beta}(z)=\frac{1}{2\pi i} \int_C \frac{e^t}{t^{\beta}-z}\, dt,\]
and the proof is similar to that of \eqref{eq:compareML}.
\end{proof}

As explained before, when estimating the difference $u-p_h u_h$ we will have to compare linear and nonlinear terms. The following result will let us gain a small power of $h$ when studying the difference $L_t -p_h L_{t,h}$ by comparing their respective Fourier multipliers, which were defined in \eqref{eq:defcompareIVP}.

\begin{prop} For $0<h<h_0$ small enough, $\alpha \in (1,2)$, $\beta\in (\frac{1}{2},1)$ and all $(t,\xi)\in (0,\infty)\times\R$, we have that
\begin{equation}\label{eq:comparesymbols}
 | E_{\beta}(i^{-\beta} t^{\beta} |\xi|^{\alpha}) - E_{\beta}(i^{-\beta} t^{\beta} h^{-\alpha} w(h\xi))|
 \lesssim \langle t\rangle\, h^{2-\alpha} \, \max\{ |\xi|^{2+\sigma-\alpha}, |\xi|^2\}.
\end{equation}
\end{prop}
\begin{proof}
We assume that $w$ is defined in $\R$ instead of $[-\pi,\pi]$ by taking its periodization. We differentiate two cases: $|\xi|\ll \frac{1}{h}$ and $ \frac{1}{h}\lesssim |\xi|$.

{\bf Step 1.} When $ |\xi|\ll \frac{1}{h}$ and $t^{\beta}|\xi|^{\alpha}\lesssim 1$, we may use the power series defining the Mittag-Leffler function, \eqref{eq:defML}, to write: 
\[ | E_{\beta}(i^{-\beta} t^{\beta} |\xi|^{\alpha}) - E_{\beta}(i^{-\beta} t^{\beta} h^{-\alpha} w(h\xi))| =\O\left(
t^{\beta}\, \left[|\xi|^{\alpha}- h^{-\alpha}\, w(h\xi)\right] \right).\]

From \Cref{thm:aboutw} we know that
\[ w(\xi)= |\xi|^{\alpha} + \O(|\xi|^2),\]
which yields
\[ | E_{\beta}(i^{-\beta} t^{\beta} |\xi|^{\alpha}) - E_{\beta}(i^{-\beta} t^{\beta} h^{-\alpha} w(h\xi))| =\O \left( 
t^{\beta} h^{2-\alpha} |\xi|^2 \right).\]

Now consider the case $1\lesssim t^{\beta} |\xi|^{\alpha}$. We focus on the top order of \eqref{eq:compareML}:
\[ \Big | e^{-it|\xi|^{\sigma}} - e^{-i t h^{-\sigma} w(h\xi)^{1/\beta}}\Big |\lesssim  t \, \Big | |\xi|^{\sigma} -h^{-\sigma} w(h\xi)^{1/\beta}\Big | \]
Using that $h|\xi|\ll 1$ and \Cref{thm:aboutw}, we obtain:
\[ \Big | |\xi|^{\sigma} -h^{-\sigma} w(h\xi)^{1/\beta}\Big | \leq \frac{1}{\beta} h^{2-\alpha} |\xi|^{2+\sigma-\alpha} + \O (|\xi|^{4+\sigma-2\alpha} h^{4-2\alpha}).\] 

The error in \eqref{eq:compareML} is controlled by:
\begin{align*}  
 \Big | \frac{t^{\beta} h^{-\alpha} w(h\xi)-t^{\beta}|\xi|^{\alpha} }{t^{2\beta}|\xi|^{\alpha}\, h^{-\alpha}\, w(h\xi)}\Big | & \lesssim t^{-\beta} h^{2-\alpha} |\xi|^{2-2\alpha} + \O \left(t^{-\beta} |\xi|^{2-\alpha} h^{2-\alpha}\right)\\
 & \lesssim t^{\beta}\, h^{2-\alpha} |\xi|^2 + \O \left(t^{-\beta} |\xi|^{2-\alpha} h^{2-\alpha}\right)
 \end{align*}

{\bf Step 2.} Suppose that $1\lesssim h|\xi|$. When $t^{\beta} |\xi|^{\alpha} \lesssim 1$, the power series in \eqref{eq:defML} yields
\begin{align*}
 | E_{\beta}(i^{-\beta} t^{\beta} |\xi|^{\alpha}) - E_{\beta}(i^{-\beta} t^{\beta} h^{-\alpha} w(h\xi))| & = 
\O \left( t^{\beta}\, \left[ |\xi|^{\alpha}- h^{-\alpha} w(h\xi)\right] \right)\\
& = \O \left( t^{\beta} |\xi|^{\alpha} \right) \lesssim  t^{\beta} |\xi|^{\alpha} (h|\xi|)^{2-\alpha}.
\end{align*}

When $1\lesssim t^{\beta}|\xi|^{\alpha}$, we use \eqref{eq:compareML} again:
\[ \Big | e^{-it|\xi|^{\sigma}} - e^{-i t h^{-\sigma} w(h\xi)^{1/\beta}}\Big | \lesssim  1 \lesssim h^{2-\alpha} |\xi|^{2-\alpha}\lesssim h^{2-\alpha} \, |\xi|^2 \, t^{\beta}.\]

The error in \eqref{eq:compareML} admits the following bound
\begin{align*}
 \Big | \frac{t^{\beta} h^{-\alpha} w(h\xi)-t^{\beta}|\xi|^{\alpha} }{t^{2\beta}|\xi|^{\alpha}\, h^{-\alpha}\, w(h\xi)}\Big |
 & \lesssim \frac{1}{t^{\beta}|\xi|^{\alpha}}\lesssim 1 \lesssim (h|\xi|)^{2-\alpha} \lesssim 
 h^{2-\alpha} |\xi|^{2} t^{\beta}.
 \end{align*}
\end{proof}

Sometimes we will need to compare the multipliers to first order, and thus we need to study their derivatives too.

\begin{prop} For $0<h<h_0$ small enough, $\alpha \in (1,2)$, $\beta\in (\frac{1}{2},1)$ and all $(t,\xi)\in (0,\infty)\times\R$, we have that
\begin{multline}\label{eq:comparesymbols2}
 \Big | C_{\alpha} |\xi|^{\alpha-1}\, E_{\beta,\beta}(i^{-\beta} t^{\beta} |\xi|^{\alpha}) - h^{1-\alpha}\, w'(h\xi)\, E_{\beta,\beta}(i^{-\beta} t^{\beta} h^{-\alpha} w(h\xi))\Big | \\ \lesssim 
\langle t\rangle^{2-\beta}\, h^{2-\alpha}\, \max\{|\xi|,|\xi |^{1+\sigma-\alpha}\} .
\end{multline}
\end{prop}
\begin{proof}
By \eqref{eq:derw}, we know that
\[ h^{1-\alpha} \, w'(h\xi)= C_{\alpha}\, |\xi|^{\alpha-1} + \O (h^{2-\alpha} |\xi|) \qquad \mbox{as}\ h|\xi|\rightarrow 0.\]
Therefore
\begin{multline*}
C_{\alpha} |\xi|^{\alpha-1}\, E_{\beta,\beta}(i^{-\beta} t^{\beta} |\xi|^{\alpha}) -   h^{1-\alpha}\, w'(h\xi)\, E_{\beta,\beta}(i^{-\beta} t^{\beta} h^{-\alpha} w(h\xi))=C_{\alpha} |\xi|^{\alpha-1}\, E_{\beta,\beta}(i^{-\beta} t^{\beta} |\xi|^{\alpha})\\
 -   C_{\alpha} |\xi|^{\alpha-1}\, E_{\beta,\beta}(i^{-\beta} t^{\beta} h^{-\alpha} w(h\xi)) - \O \left( h^{2-\alpha} |\xi| E_{\beta,\beta} (i^{-\beta} t^{\beta} h^{-\alpha} w(h\xi)) \right)
\end{multline*}
Using \eqref{eq:genMLasymp}, we know that
\[ \O \left( h^{2-\alpha} \,|\xi|\, E_{\beta,\beta} (i^{-\beta} t^{\beta} h^{-\alpha} w(h\xi)) \right)= \O \left( h^{2-\alpha} \, |\xi|\, \langle \xi\rangle^{\sigma-\alpha} \right),\]
which is controlled by the error in \eqref{eq:comparesymbols2}. 

Therefore we only need to prove the desired bound for 
\[ E_{\beta,\beta}(i^{-\beta} t^{\beta} |\xi|^{\alpha})-  E_{\beta,\beta}(i^{-\beta} t^{\beta} h^{-\alpha} w(h\xi)).\]
The rest of the proof is very similar to that of \eqref{eq:comparesymbols}, and depends exclusively on \eqref{eq:comparegenML} and \Cref{thm:aboutw}, so we omit it.
\end{proof}

We finally put to good use these two propositions.

\begin{prop}\label{thm:compareL2_1} Let $0<h<h_0$ small enough, $\alpha \in (1,2)$, and $\beta\in (\frac{1}{2},1)$. Then for any  small $\varepsilon>0$ there exists some small $\bar{b}=\bar{b}(\varepsilon)>0$ such that
\begin{align*}
\norm{ L_t f -p_h L_{t,h} \Pi_h f_{2h} }_{L^2_x} & \lesssim \langle t\rangle \, h^{\bar{b}} \,\norm{f}_{H^{\varepsilon}_x},\\
\norm{\int_0^t \left[ N_{t-t'} p_h g_h(t') - p_h N_{t-t',h} g_h (t') \right]\, dt'}_{L^{\infty}_T L^2_x}& \lesssim_T h^{\bar{b}}\, \norm{g_h}_{L^2_T H^{\sigma-\alpha+\varepsilon}_h},
\end{align*}
where the implicit constant in the second equation is controlled by a power of $T$.
\end{prop}
\begin{rk} A similar proof combined with complex interpolation yields:
\[ \norm{ L_t f -p_h L_{t,h} \Pi_h f_{2h} }_{H^{s_1}_x} \lesssim h^{(s_2-s_1)\varepsilon} \,\norm{f}_{H^{s_2}_x},\]
where $s_2\geq s_1$, $s_1,s_2\in [0,1]$, $\varepsilon=\varepsilon(\alpha,\beta)>0$ and the implicit constant depends on a power of $\langle T\rangle$.
\end{rk}
\begin{proof}
We explain how to prove the first estimate, since the ideas involved in the second are similar. The main tools are the discrete $L^{\infty}_T L^2_h$ estimates, (\Cref{thm:massestimates}), their continuous counterpart  \Cref{thm:main_cont_est}, and \eqref{eq:comparesymbols}-\eqref{eq:comparesymbols2}.

{\bf Step 1.} The following decomposition of $L_t f-p_h L_{t,h} \Pi_h f_{2h}$ will be used in many upcoming proofs to exploit cancellation between these operators. We write:
\begin{align*}
p_h L_{t,h} \Pi_h f_{2h}- L_t f  = & p_h L_{t,h} \Pi_h f_{2h}- L_t p_h \Pi_h f_{2h} +L_t (p_h \Pi_h f_{2h} -f) \\
 = & (p_h L_{t,h} \Pi_h f_{2h}- L_t p_h \Pi_h f_{2h}) \ast \varphi_1  \\
 & + (p_h L_{t,h} \Pi_h f_{2h}- L_t p_h \Pi_h f_{2h}) \ast \varphi_2 \\
 & + L_t (p_h \Pi_h f_{2h} -f)\\
= &\  I + I\! I + I\! I\! I,
\end{align*}
where $\varphi_1$ is a smooth function whose Fourier transform is supported in the region $\{\xi\in\R \mid |\xi|\lesssim h^{-b}\}$, and $\varphi_2$ is a smooth function whose Fourier transform is supported in the region $\{\xi\in \R \mid h^{-b}\lesssim |\xi|\}$, and such that $\widehat{\varphi_1} + \widehat{\varphi_2}=1$. The constant $b>0$ will be chosen later.

By the Plancherel theorem and \eqref{eq:FTinterpolation},
\[ \norm{I}_{L^2_x}= \norm{[E_{\beta}(i^{-\beta} t^{\beta} h^{-\alpha} w(h\xi))- E_{\beta}(i^{-\beta} t^{\beta} |\xi|^{\alpha})]\, P_h(\xi) \widehat{\Pi_h f_{2h}}(h\xi)}_{L^2_{\xi}(|\xi|\lesssim h^{-b})}\]
By \eqref{eq:comparesymbols}, for any small $\varepsilon>0$ we have that
\begin{align*}
 \norm{I}_{L^2_x} & \lesssim \norm{t^{\beta} h^{2-\alpha}\max\{|\xi|^2,|\xi|^{2+\sigma-\alpha}\}\, P_h(\xi) \widehat{\Pi_h f_{2h}}(h\xi)}_{L^2_{\xi}(|\xi|\lesssim h^{-b})}\\
 & \lesssim t^{\beta} h^{2-\alpha}\, \max\{h^{2\varepsilon-2b},h^{2\varepsilon-(2+\sigma-\alpha)b}\}\, \norm{p_h \Pi_h f_{2h}}_{H^{\varepsilon}_x(\R)}\lesssim  t^{\beta} h^{\bar{b}}\, \norm{f}_{H^{\varepsilon}_x(\R)}.
 \end{align*}
 Here we see that we must choose $b<1-\alpha/2+\varepsilon$ in the region where $|\xi|\leq 1$ and $b<\frac{2\varepsilon+2-\alpha}{2+\sigma-\alpha}$ when $|\xi|>1$. The latter is more restrictive.
 
Now we look at the high frequency part. By the Plancherel theorem and \eqref{eq:FTinterpolation},
\begin{align*}
\norm{I\! I }_{L^2_x} & = \norm{ [E_{\beta}(i^{-\beta} t^{\beta} h^{-\alpha} w(h\xi))- E_{\beta}(i^{-\beta} t^{\beta} |\xi|^{\alpha})]\, P_h(\xi) \widehat{\Pi_h f_{2h}}(h\xi)}_{L^2_{\xi}(h^{-b}\lesssim |\xi|)}\\
& \lesssim h^{\varepsilon b}\, \norm{|\xi|^{\varepsilon}\,[ E_{\beta}(i^{-\beta} t^{\beta} h^{-\alpha} w(h\xi))- E_{\beta}(i^{-\beta} t^{\beta} |\xi|^{\alpha})]\, P_h(\xi)\,\widehat{\Pi_h f_{2h}}(h\xi)}_{L^2_{\xi}(h^{-b}\lesssim |\xi|)}\\
& \lesssim h^{\varepsilon b}\, \norm{f}_{H^{\varepsilon}(\R)},
\end{align*}
thanks to the fact that $E_{\beta}$ is uniformly bounded (in other words, each operator $p_h$, $L_t$ and $L_{t,h}$ is bounded).

Finally, we tackle the last term
\begin{align*}
\norm{ I\! I\! I}_{L^2_x} & = \norm{ L_t (p_h \Pi_h f_{2h} -f)}_{L^2_x}\\
& \leq \norm{p_h \Pi_h f_{2h} -f}_{L^2_x}= \norm{p_{2h} f_{2h} -f}_{L^2_x}\leq h^{s} \norm{f}_{H^s_x}
\end{align*}
after using \Cref{thm:gainh} and the fact that $p_h \Pi_h f_{2h}=p_{2h} f_{2h}$, which is easy to check given that both $p_h$ and $\Pi_h$ are linear interpolators. 
%Let us explain this claim, for $x\in [mh,(m+1)h)$ with $m$ even
%\begin{align*}
% (p_h \Pi_h f_{2h})(x) & = (\Pi_h f_{2h})(mh) + \frac{(\Pi_h f_{2h})((m+1)h)- (\Pi_h f_{2h})(mh)}{h} (x-mh)\\
%& = f_{2h}(mh) + \frac{\frac{f_{2h}(mh)+f_{2h}((m+2)h)}{2}- f_{2h}(mh)}{h} (x-mh)\\
%& = f_{2h}(mh) + \frac{f_{2h}((m+2)h)- f_{2h}(mh)}{2h} (x-mh)= p_{2h} f_{2h}(x).
% \end{align*}
% The case of $m$ odd is similar.
% \begin{align*}
% (p_h \Pi_h f_{2h})(x) & = (\Pi_h f_{2h})(mh) + \frac{(\Pi_h f_{2h})((m+1)h)- (\Pi_h f_{2h})(mh)}{h} (x-mh)\\
%& = \frac{f_{2h}((m-1)h)+f_{2h}((m+1)h)}{2} + \frac{f_{2h}((m+1)h)- \frac{f_{2h}((m-1)h)+f_{2h}((m+1)h)}{2}}{h} (x-mh)\\
%& = \frac{f_{2h}((m-1)h)+f_{2h}((m+1)h)}{2} + \frac{f_{2h}((m+1)h)- f_{2h}((m-1)h)}{2h} (x-mh)\\
%& = f_{2h}((m-1)h)+ \frac{f_{2h}((m+1)h)- f_{2h}((m-1)h)}{2h} (x-(m-1)h)= p_{2h}f_{2h}(x).
% \end{align*}
\end{proof}

The main idea in the proof before was dividing $\R$ into two regions: a low frequency part and a high frequency part, which depend on $h$. In the high frequency region, the Plancherel theorem allows us to trade derivatives in order to gain a small power of $h$. In that step, it is critical that we are in $L^2_x$, and therefore it is unlikely that we will be able to generalize that to other norms such as $L^p_x L^q_T$. For that reason we introduce a cut-off to low frequencies. We define $K_h^b$ as the Fourier multiplier operator given by 
\begin{equation}\label{eq:defK}
 \widehat{K_h^b f}(\xi) = \chi_{|\xi|\lesssim h^{-b}} \, \widehat{f}(\xi),
 \end{equation}
 where $0\leq \chi\leq 1$ is a smooth function supported in $\{\xi\in\R \mid |\xi|\lesssim h^{-b}\}$ with value 1 in a slightly smaller (comparable) region. We will also use the notation $\varphi_1 = \chi^{\vee}$. In \eqref{eq:defK}, the value of $b>0$ will be fixed later, and might vary in the upcoming results, but its specific value is unimportant for our analysis.  Note however, that this value will be related to the rate of convergence to the continuum limit, so it might be interesting to track it for computational purposes.
 
All in all, the goal is to study $L_t f - K_h^b p_h L_{t,h} \Pi_h f_{2h}$. Note that in \Cref{thm:compareL2_1} we also proved:
 \begin{prop}\label{thm:compareL2}
Let $0<h<h_0$ small enough, $\alpha \in (1,2)$, and $\beta\in (\frac{1}{2},1)$. Then for any small $\varepsilon>0$, there exists some $b,\bar{b}>0$ satisfying
\[ 0<b< \frac{2-\alpha}{2+\sigma-\alpha-\varepsilon}\]
 such that
\[ \norm{ L_t f - K_h^b p_h L_{t,h} \Pi_h f_{2h} }_{L^2_x}  \lesssim \langle t\rangle \, h^{\bar{b}} \,\norm{f}_{H^{\varepsilon}_x}.\]
Moreover, there exists some $b,\bar{b}>0$ such that
\[  \norm{\int_0^t \left[ N_{t-t'} K_h^b p_h g_h(t') - K_h^b p_h N_{t-t',h} g_h (t') \right]\, dt'}_{L^{\infty}_T L^2_x} \lesssim_T h^{\bar{b}}\, \norm{g_h}_{L^2_T H^{\sigma-\alpha}_h}.\]
\end{prop}
\begin{rk}
Note that we do not need to lose derivatives in the estimate of the nonlinearity, since we are in the low-frequency regime. However, this loss will inevitably happen later in the high-frequency region, so there is no reason to improve the low-frequency estimate anyway.
\end{rk}

We now see how to gain a small power of $h$ in the estimates involving the smoothing effect.

\begin{prop}\label{thm:comparesmoothing} Let $0<h<h_0$ small enough, $\alpha \in (1,2)$, $\beta\in (\frac{1}{2},1)$, $\sigma=\frac{\alpha}{\beta}$ and $\gamma=\frac{\sigma-1}{2}$. Then for any small $\varepsilon>0$, there exists some $b,\bar{b}>0$ satisfying
\[ 0<b< \frac{2-\alpha}{\sigma-\alpha+\gamma-\varepsilon+\frac{5}{2}}\]
such that
\[ \norm{\langle \nabla\rangle^{\gamma-} \, (L_t f -K_h^b p_h L_{t,h} \Pi_h f_{2h}) }_{L^{\infty}_x L^{2}_T} \lesssim_T h^{\bar{b}}\, \norm{f}_{H^{\varepsilon}_x}.\]
where the implicit constant is controlled by a power of $T$. Moreover, there exists some $b,\bar{b}>0$ such that
\[  \norm{\langle \nabla\rangle^{\gamma-} \,  \int_0^t \left[ N_{t-t'} K_h^b p_h g_h(t') - K_h^b p_h N_{t-t',h} g_h (t') \right]\, dt'}_{L^{\infty}_x L^2_T} \lesssim_T h^{\bar{b}}\, \norm{g_h}_{L^2_T H^{\sigma-\alpha}_h}.\]
\end{prop}
\begin{proof}
We focus on the first estimate, since the second estimate is simpler.

{\bf Step 1.}
As before, we decompose:
\begin{align*}
K_h^b p_h L_{t,h} \Pi_h f_{2h}- L_t f  = & K_h^b p_h L_{t,h} \Pi_h f_{2h}- L_t p_h \Pi_h f_{2h} +L_t (p_h \Pi_h f_{2h} -f) \\
 = & K_h^b (p_h L_{t,h} \Pi_h f_{2h}- L_t p_h \Pi_h f_{2h})  \\
 & - (1-K_h^b)\,( L_t p_h \Pi_h f_{2h}) \\
 & +L_t (p_h \Pi_h f_{2h} -f)\\
= &\  I + I\! I + I\! I\! I.
\end{align*}
As we explained before, the constant $b>0$ will be chosen later. We will also assume that our data is supported $|\xi|\geq 1$ so that we can substitute $\langle \nabla\rangle^{\gamma}$ by $|\nabla|^{\gamma}$. The case $|\xi|\leq 1$ is easier and will be discussed at the end of the proof. Fix some small $\varepsilon>0$. Firstly, we use \eqref{eq:comparesymbols}, the Minkowski inequality, and the Cauchy-Schwartz inequality:
\begin{align*}
\norm{|\nabla|^{\gamma}\, I}_{L^{\infty}_x L^2_T} & \lesssim \norm{\int_{|\xi|\lesssim h^{-b}} |P_h(\xi)\, \widehat{\Pi_h f_{2h}}(h\xi)|\, \langle t\rangle |\xi|^{2+\sigma-\alpha+\gamma} h^{2-\alpha} \, d\xi}_{L^{\infty}_x L^2_T}\\
& \lesssim \langle T\rangle^{3/2}\, h^{2-\alpha}\, \int_{|\xi|\lesssim h^{-b}} |P_h(\xi)\, \widehat{\Pi_h f_{2h}}(h\xi)|\,|\xi|^{2+\sigma-\alpha+\gamma}\, d\xi\\
& \lesssim \langle T\rangle^{3/2}\, h^{2-\alpha}\,\norm{p_h \Pi_h f_{2h}}_{H^{\varepsilon}_x}\, \left(\int_{|\xi|\lesssim h^{-b}} |\xi|^{2(\gamma+2+\sigma-\alpha-\varepsilon)}\, d\xi\right)^{1/2}\\
& \lesssim  \langle T\rangle^{3/2}\, h^{2-\alpha}\, h^{-b (\sigma-\alpha+\gamma-\varepsilon+\frac{5}{2})}\, \norm{f}_{H^{\varepsilon}_x}.
\end{align*}
By choosing $b$ small enough, we can obtain a positive exponent for $h$.

{\bf Step 2.} Now we consider the high-frequency part. By the triangle inequality,
\[\norm{\langle\nabla\rangle^{\gamma}\, I\! I}_{L^{\infty}_x L^2_T} = \norm{\langle\nabla\rangle^{\gamma}\,L_t\, (1-K_h^b)\, p_h \Pi_h f_{2h} }_{L^{\infty}_x L^2_T}.\]
We use the smoothing effect of the operator $L_t$, see \Cref{thm:main_cont_est}:
\[\norm{\langle\nabla\rangle^{\gamma}\, I\! I}_{L^{\infty}_x L^2_T}\lesssim \norm{(1-K_h^b)\,p_h \Pi_h f_{2h}}_{L^2_x}\lesssim h^{b\varepsilon}\, \norm{f}_{H^{\varepsilon}_x}.\]

{\bf Step 3.} Finally, we consider the last term, where we may again use the smoothing effect for $L_t$, \Cref{thm:main_cont_est}:
\begin{align*}
 \norm{\langle\nabla\rangle^{\gamma}I\! I\! I}_{L^{\infty}_x L^2_T} & = \norm{\langle\nabla\rangle^{\gamma} L_t (p_h \Pi_h f_{2h} -f)}_{L^{\infty}_x L^2_T} \\
 & \lesssim_T \norm{p_h \Pi_h f_{2h} -f}_{L^2_x} \lesssim_T h^{\varepsilon}\, \norm{f}_{H^{\varepsilon}_x},
 \end{align*}
where the last inequality follows from  \Cref{thm:gainh}.

Before finishing the proof, we explain how to deal with the case where our data is supported in $|\xi|\leq 1$. We can mimic Step 1 and estimate $|\xi|\leq 1$ directly to obtain the desired bound. Step 2 is not necessary since the regions $\{h^{-b}\lesssim |\xi|\}$ and $\{|\xi|\leq 1\}$ have empty intersection for $h$ small enough, and Step 3 works equally well for $|\xi|\leq 1$, as shown in \Cref{thm:main_cont_est}.
\end{proof}

Finally, we show how to estimate $L_t-p_h L_{t,h}$ in the maximal norm involved in the local well-posedness theory. While the two previous norms shared many similarities and only required \eqref{eq:comparesymbols}, we will now see that \eqref{eq:comparesymbols2} is necessary in this case.

\begin{prop}\label{thm:comparemaximal}Let $0<h<h_0$ small enough, $\alpha \in (1,2)$, $\beta\in (\frac{1}{2},1)$, $s=\frac{1}{2}-\frac{1}{p}$ and $p\geq 4$. Then for any small $\varepsilon>0$, there exists some $b,\bar{b}>0$ such that
\begin{align*}
\norm{L_t f - K_h^b p_h L_{t,h} f_h}_{L^{p}_x L^{\infty}_T} & \lesssim_T h^{\bar{b}}\, \norm{f}_{H^{s+\varepsilon}_x},
\end{align*}
where the implicit constant is controlled by a power of $T$. 

Moreover, there exists some $b,\bar{b}>0$ such that
\[ \norm{\int_0^t \left[ N_{t-t'}K_h^b  p_h g_h(t') - K_h^b p_h N_{t-t',h} g_h (t') \right]\, dt'}_{L^{p}_x L^{\infty}_T}  \lesssim_T h^{\bar{b}}\, \norm{g_h}_{L^2_T H^{s+\sigma-\alpha}_x}\]
\end{prop}
\begin{proof}
We will only prove the first estimate, since the ideas involved in the second are similar. The main tools to prove this result are the continuous maximal function estimate (\Cref{thm:main_cont_est}), its discrete counterpart (\Cref{thm:maximalestimates}), and \eqref{eq:comparesymbols}-\eqref{eq:comparesymbols2}.

Once again we write:
\begin{align*}
K_h^b p_h L_{t,h} f_h- L_t f  = & K_h^b p_h L_{t,h} f_h- L_t p_h f_h +L_t (p_h f_h -f) \\
 = & K_h^b (p_h L_{t,h} f_h- L_t p_h f_h)  \\
 & - (1-K_h^b)\,( L_t p_h f_h) \\
 & +L_t (p_h f_h -f)\\
= &\  I + I\! I + I\! I\! I.
\end{align*}
As we explained before, the constant $b>0$ will be chosen later.

{\bf Step 1.} We first focus on the region $1\lesssim t^{\beta}|\xi|^{\alpha}$. For $s\in (0,1)$, define:
\begin{equation}\label{eq:maximal_difference}
|\nabla|^{-s} \mathcal{T}(t,h,x) =\int_{\R} e^{ix\xi}\,|\xi|^{-s}\, \chi_{|\xi|\lesssim h^{-b}} \,\chi_{1\lesssim t^{\beta}|\xi|^{\alpha}}
 \left[ E_{\beta} (i^{-\beta} \, t^{\beta}\, h^{-\alpha}\, w(h\xi))- E_{\beta} (i^{-\beta} \, t^{\beta}\, |\xi|^{\alpha})\right]\, d\xi.
\end{equation}
 
Let
\begin{align*}
\Omega_1 & := \{ \xi\in\R \mid |\xi|\leq |x|^{-1}\}\\
\Omega_2 & := \{ \xi\in\R-\Omega_1 \mid |x-t\sigma|\xi|^{\sigma-1}|\leq \frac{\sigma-1}{2\sigma} |x|\}\\
\Omega_3 & := \R-\Omega_1\cup\Omega_2.
\end{align*}

Our goal is to show that 
\begin{equation}\label{eq:strongerKdVtypedecay}
||\nabla|^{-s} \mathcal{T}(t,x)|\lesssim_T h^{\bar{b}} \, |x|^{s-1},
\end{equation}
for some $\bar{b}>0$, uniformly in $t\in [0,T]$.

Let $I_i (t,h,x)$ to be the integral over $\Omega_i$, so that their sum is \eqref{eq:maximal_difference}. Note that
\begin{align*}
|I_1(t,h,x)| & \leq \langle T \rangle \, \int_{\Omega_1}|\xi|^{-s} \, \chi_{|\xi|\lesssim h^{-b}} \, h^{2-\alpha}\, |\xi|^{2+\sigma-\alpha}\, d\xi\\
& \leq \langle T \rangle  \, h^{2-\alpha} \, h^{-b(2+\sigma-\alpha)}\, |x|^{s-1}\leq \langle T \rangle \, h^{\bar{b}}\, |x|^{s-1}.
\end{align*}
as long as $b>0$ is chosen small enough.

In $\Omega_2$, we use \eqref{eq:compareML} to write:
\[ E_{\beta} (i^{-\beta} \, t^{\beta}\, h^{-\alpha}\, w(h\xi))- E_{\beta} (i^{-\beta} \, t^{\beta}\, |\xi|^{\alpha}) = e^{-it|\xi|^{\sigma}} \, \left( 1- e^{it|\xi|^{\sigma}-it\phi_h(h\xi)}\right) + E(t,h,\xi)\]
and we postpone estimating the error $E$ in this approximation to the next step.

We write:
\begin{equation}\label{eq:approxerror}
 I_2(t,h,x) = \int_{\Omega_2} e^{ix\xi-it|\xi|^{\sigma}} \,|\xi|^{-s}\, \chi_{|\xi|\lesssim h^{-b}} \, \chi_{1\lesssim t^{\beta}|\xi|^{\alpha}}\, \left( 1- e^{it|\xi|^{\sigma}-it\phi_h(h\xi)}\right)\, d\xi + \mathcal{E}(t,h,x).
 \end{equation}
In $\Omega_2$ we have that $t|\xi|^{\sigma-1}\sim |x|$ and $|\xi|>|x|^{-1}$. We use Van der Corput's lemma and
the bound given by \eqref{eq:comparesymbols}.
\begin{align*}
|I_2 (t,h,x)-\mathcal{E}(t,h,x)| &\lesssim \sup_{\xi\in\Omega_2}  |1- e^{it|\xi|^{\sigma}-it\phi_h(h\xi)}|\, \chi_{|\xi|\lesssim h^{-b}} \, \chi_{1\lesssim t^{\beta}|\xi|^{\alpha}} \, |\xi|^{-s} \cdot (t|\xi|^{\sigma-2})^{-1/2}\\
 &\lesssim  \sup_{\xi\in\Omega_2} \, \langle T\rangle\, h^{2-\alpha}\, |\xi|^{2+\sigma-\alpha}\, \chi_{|\xi|\lesssim h^{-b}}  \,
 |\xi|^{1/2-s} \cdot |x|^{-1}\, (t|\xi|^{\sigma-1})^{1/2}\\
 & =  \sup_{\xi\in\Omega_2}\,  \langle T\rangle^{3/2}\, h^{2-\alpha}\, |\xi|^{2+\frac{3}{2}\sigma-\alpha}\,\chi_{|\xi|\lesssim h^{-b}} \,
 |\xi|^{-s} \cdot |x|^{-1}\\
&  \lesssim \langle T\rangle^{3/2}\, h^{2-\alpha}\, h^{-b(2+\frac{3}{2}\sigma-\alpha)}\,
 |x|^{s-1}= \langle T\rangle^{3/2} \, h^{\bar{b}}\, |x|^{s-1}.
\end{align*}
Note that in this step we used to get the condition $\frac{1}{2}\leq s$ (see for instance the proof of \Cref{thm:discretemaximal}), and then a $TT^{\ast}$ argument halved the loss of derivatives. This condition on $s$ was necessary in order to use the bound $|\xi|>|x|^{-1}$. However, we do not need to do that here, because we can bound positive powers of $|\xi|$ by terms in $h^{-b}$ thanks to $K_h^b$. This allows us to avoid the $TT^{\ast}$ argument altogether, since now $s$ is allowed to range over $(0,1)$.

Finally, in $\Omega_3$ we may integrate by parts:
\[ I_3 (t,h,x)\sim \int_{\Omega_3} e^{ix\xi-it|\xi|^{\sigma}} \,\frac{d}{d\xi}\left[\frac{1}{x-t|\xi|^{\sigma-1}}\,|\xi|^{-s}\, \chi_{|\xi|\lesssim h^{-b}} \, \chi_{1\lesssim t^{\beta}|\xi|^{\alpha}}\, \left( 1- e^{it|\xi|^{\sigma}-it\phi_h(h\xi)}\right)\right]\, d\xi.\]
Recall that in $\Omega_3$ we have that $|x-t|\xi|^{\sigma-1}|\geq |x|$ and $|\xi|>|x|^{-1}$.
According to the proof of \eqref{eq:comparesymbols2}, we can bound the derivative of our symbol by
$\langle T \rangle^2 h^{2-\alpha} |\xi|^{1+\sigma-\alpha}$. Therefore,
\begin{align*}
 |I_3 (t,h,x)|& \lesssim  \langle T \rangle^2 h^{2-\alpha} \int_{\Omega_3} |\xi|^{1+\sigma-\alpha} \, |\xi|^{-s}\,\chi_{|\xi|\lesssim h^{-b}} \,d\xi\\
&\lesssim  \langle T \rangle^2 h^{2-\alpha} \, h^{-b(2+\sigma-\alpha)}\, \int_{\Omega_3} |\xi|^{-s-1} \, d\xi\lesssim  \langle T \rangle^2 h^{\bar{b}} \, |x|^{s-1}
\end{align*}
as long as $b>0$ is small enough.

{\bf Step 2.} Now we estimate the error in \eqref{eq:approxerror}. As shown in the proof of \eqref{eq:comparesymbols},
\[ E(h,t,\xi)= t^{\beta} h^{2-\alpha} |\xi|^{2-\alpha} + \O \left(|\xi|^2 h^{2-\alpha}\right)\]
in the region $h|\xi|\ll 1$, which clearly contains our region $h^{b} |\xi|\lesssim 1$.

As before, we define 
\[ \mathcal{E}(t,h,x) =\int_{\R} e^{ix\xi}\,|\xi|^{-s}\,\chi_{|\xi|\lesssim h^{-b}} \, \chi_{1\lesssim t^{\beta}|\xi|^{\alpha}}\, E(h,t,\xi)\,d\xi\]
and our goal is to show 
\[| \mathcal{E}(t,h,x)|\lesssim_T h^{\bar{b}} \, |x|^{s-1}.\]

When $|x|$ is small, we can integrate directly:
\[
 |\mathcal{E}(t,h,x)|  \leq \int_{\R} |\xi|^{-s}\, \chi_{|\xi|\lesssim h^{-b}}\, \chi_{1\lesssim t^{\beta}|\xi|^{\alpha}} \, h^{2-\alpha} |\xi|^{2-\alpha}\, d\xi \lesssim h^{2-\alpha} \, h^{-b(3-\alpha-s)} \lesssim h^{\bar{b}} ,
\]
as long as $b>0$ is small. When $|x|>2$, we integrate by parts and obtain
\begin{align*}
 |\mathcal{E}(t,h,x)| & \leq |x|^{-1}\,\int_{\R} \Big | \frac{d}{d\xi} \left(|\xi|^{-s}\, \chi_{|\xi|\lesssim h^{-b}}  \, \chi_{1\lesssim t^{\beta}|\xi|^{\alpha}} \,E(t,h,\xi)\right)\Big | \, d\xi .
\end{align*}
The top-order term occurs when the derivative hits $E$, but as shown in the proof of \eqref{eq:comparesymbols2}, this only contributes $\O (t^{\beta} h^{2-\alpha}\,|\xi|)$. Therefore, we obtain 
\[ |\mathcal{E}(t,h,x)|\lesssim \langle T\rangle^2\, h^{\bar{b}}\,|x|^{-1} \qquad \mbox{for}\ |x|>2.\]
Combining the bounds for small and large $|x|$, we have 
\[  |\mathcal{E}(t,h,x)|\lesssim  \langle T\rangle^2 \, h^{\bar{b}}\, |x|^{s-1}.\]

{\bf Step 3.} Now we consider the low frequency case, $t^{\beta}|\xi|^{\alpha}\lesssim 1$. Let
\[
 |\nabla|^{-s} \mathcal{U}(t,h,x) =\int_{\R} e^{ix\xi}\,|\xi|^{-s}\, \chi_{|\xi|\lesssim h^{-b}} \, \chi_{t^{\beta}|\xi|^{\alpha}\lesssim 1}
  \left[ E_{\beta} (i^{-\beta} \, t^{\beta}\, h^{-\alpha}\, w(h\xi))- E_{\beta} (i^{-\beta} \, t^{\beta}\, |\xi|^{\alpha})\right]\, d\xi.
\]
When $|x|\leq 2$, we can use \eqref{eq:comparesymbols} to get
\begin{align*}
 ||\nabla|^{-s} \mathcal{U}(t,h,x)| & \leq \int_{\R}|\xi|^{-s} \, \chi_{t^{\beta}|\xi|^{\alpha}\lesssim 1}\,\chi_{|\xi|\lesssim h^{-b}} \,\langle t\rangle h^{2-\alpha}\, |\xi|^{2+\sigma-\alpha}\, d\xi\\
 & \lesssim \langle T\rangle\, h^{2-\alpha}\, h^{-b(2+\sigma-\alpha)} \int_{\R}|\xi|^{-s} \, \chi_{t^{\beta}|\xi|^{\alpha}\lesssim 1}\,d\xi \lesssim \langle T\rangle^2 \, h^{\bar{b}}.
 \end{align*}
When $|x|>2$, we can use integration by parts:
\begin{multline*}
 ||\nabla|^{-s} \mathcal{U}(t,h,x)|\leq \\
 |x|^{-1}\, \int_{\R} \Big | \frac{d}{d\xi} \left( |\xi|^{-s}\, \chi_{|\xi|\lesssim h^{-b}}  \, \chi_{t^{\beta}|\xi|^{\alpha}\lesssim 1}\, \left[ E_{\beta} (i^{-\beta} \, t^{\beta}\, h^{-\alpha}\, w(h\xi))- E_{\beta} (i^{-\beta} \, t^{\beta}\, |\xi|^{\alpha})\right]\right)\Big |\, d\xi
 \end{multline*}
 and use \eqref{eq:comparesymbols2} to control the derivative as we did before. What follows is standard so we omit it.
 
Finally, the decay given by \eqref{eq:strongerKdVtypedecay} together with the Hardy-Littlewood-Sobolev inequality prove the desired estimate for $I$.

{\bf Step 4.} We now study the term 
\[I\! I=(1-K_h^b)\,( L_t p_h f_h).\]
As before, we use \Cref{thm:main_cont_est} to yield the desired inequality, and we trade derivatives for powers of $h$ once we are in $L^2_x$.

Finally, we control $I\! I\! I=L_t (p_h f_h -f)$ with the help of \Cref{thm:main_cont_est}, which yields:
\[\norm{|\nabla|^{-s} L_t (p_h f_h -f)}_{L^p_x L^{\infty}_T}\lesssim \norm{p_h f_h -f}_{L^2_x}\lesssim h^{\varepsilon} \, \norm{f}_{H^{\varepsilon}_x}.
\]
\end{proof}

Finally, one can interpolate between the estimates in $L^{\infty}_x  L^2_T$ and $L^{2(p-1)}_x L^{\infty}_T$ following the standard approach in \cite{KPV}. We therefore omit the proof and refer the reader to \cite{interpolation} and \cite[Section~2.2.5]{mythesis} for additional details. We highlight that, despite the many operators applied to $f$, the map $f\mapsto L_t f - K_h^b p_h L_{t,h} \Pi_h f_{2h}$ is linear and continuous in the corresponding spaces. 
\begin{prop}\label{thm:compareint}
Let $0<h<h_0$ small enough, $\alpha \in (1,2)$, $\beta\in (\frac{1}{2},1)$, $\sigma=\frac{\alpha}{\beta}$, $\gamma=\frac{\sigma-1}{2}$ and $\tilde{\gamma}=\alpha-\frac{\sigma+1}{2}$. For $p\geq 3$, let $s=\frac{1}{2}-\frac{1}{2(p-1)}$, $0\leq \gamma'<\gamma$ and $0\leq \tilde{\gamma}'<\tilde{\gamma}$. Then for any small $\varepsilon>0$, there exists some $b,\bar{b}>0$ such that
\begin{align*}
\norm{L_t f - K_h^b p_h L_{t,h} \Pi_h f_{2h}}_{L^{4(p-1)}_x L^{4}_T} & \lesssim_T h^{\bar{b}}\, \norm{f}_{H^{\frac{s-\gamma'}{2}+\varepsilon}_x}.
\end{align*}
where the implicit constant is controlled by a power of $T$. Moreover, there exists some $b,\bar{b}>0$ such that 
\[ \norm{\int_0^t \left[ N_{t-t'}K_h^b  p_h g_h(t') - K_h^b p_h N_{t-t',h} g_h (t') \right]\, dt'}_{L^{4(p-1)}_x L^{4}_T}  \lesssim_T h^{\bar{b}}\, \norm{g_h}_{L^2_T H^{\frac{s+\sigma-\alpha-\tilde{\gamma}'}{2}+\varepsilon}_x}.\]
\end{prop}

\subsection{Main argument}

We are now ready to give the main argument of the continuum limit. We will work in the space where we developed the local well-posedness theory. Consider the norms:
\begin{align}
\eta_1 (v) & :=\norm{ \langle\nabla\rangle^{s+\sigma-\alpha} \,v}_{L^{\infty}_x L^2_T},\nonumber\\
\eta_2 (v) & := \norm{ \langle\nabla\rangle^s\,v}_{L^{\infty}_T L^2_x},\label{eq:norms}\\
\eta_3 (v) & := \norm{v}_{ L^{2(p-1)}_x L^{\infty}_T},\nonumber\\
\eta_4 (v) & := \norm{\langle\nabla\rangle^{(s+\sigma-\alpha)/2} v}_{ L^{4(p-1)}_x L^{4}_T},\nonumber
\end{align}
and define 
\[ \Lambda_T (v):=\max_{j=1,2,3,4} \eta_j(v).\]
We define the space
\[ X_{T}^s = \{ v \in C( [0,T], H^s_x(\R))\mid \Lambda_T (v)<\infty \} \]
and its discrete counterpart $X_{T,h}^s$.

Consider the continuous problem given by \eqref{eq:contequation}. By \Cref{thm:contLWP}, for any $\tilde{s}\geq \frac{1}{2}-\frac{1}{2(p-1)}$, we have a unique solution to the continuous problem, $u\in X_{T}^{\tilde{s}}$, defined in some time interval $[0,T]$ depending on the norm of the initial data $\norm{f}_{H^{\tilde{s}}(\R)}$. By \Cref{thm:discretelwp}, we also have a unique solution to the discrete problem given by \eqref{eq:inhomogeneous2}, $u_h\in X_{T,h}^{s}$, in some time interval $[0,T]$ depending on the norm of the initial data $\norm{f}_{H^{s}(\R)}$. In both cases, the existence of a solution is guaranteed via a fixed point argument which consists of showing that:
\begin{equation}\label{eq:fixed_point}
\Lambda_T (u) \lesssim \norm{f}_{H^s(\R)} + T^{0+} \, \norm{\langle \nabla\rangle^{s+\sigma-\alpha} ( |u|^{p-1}\, u )}_{L^2_{T,x}} \lesssim  \norm{f}_{H^s(\R)} + T^{0+} \,\Lambda_T (u)^p.
\end{equation}
This nonlinear estimate directly follows from the local well-posedness theory, which in turn follows from the estimates in \Cref{thm:main_cont_est}. See also \cite[Lemma~2.4.1]{mythesis} or \cite{mypaper} for a detailed proof.

Without loss of generality, we may assume a common interval of existence for both problems which depends on $\norm{f}_{H^{\tilde{s}}(\R)}$, assumed to be finite, where $\tilde{s}>s$ is fixed in \Cref{thm:blabla} below. Consequentely, we have that $u_h(t)$ and $u(t)$ satisfy the initial value problems:
\begin{align*}
 u(t) & =L_t f \pm i^{-\beta}\, \int_0^t N_{t-t'} (|u(t')|^{p-1} u(t')) \, dt'\\
u_h (t) & = L_{t,h} \Pi_h f_{2h} \pm i^{-\beta} \, \int_0^t N_{t-t',h} \Pi_h R_h (|u_h(t')|^{p-1} u_h(t')) \, dt'
\end{align*}
in their respective spaces $C([0,T],H^{\tilde{s}}_x(\R))$ and $C([0,T],H^{\tilde{s}}_h)$, where we can take any regularity $\tilde{s}>\frac{1}{2}-\frac{1}{2(p-1)}$ that we wish. Recall that the linear and nonlinear operators $L_t$, $L_{t,h}$, $N_t$ and $N_{t,h}$ are defined in \eqref{eq:defcompareIVP}.

We take the linear interpolation of $u_h$ in order to work in the common space $C([0,T],H^{\tilde{s}}(\R))$:
\begin{align}\label{eq:IVPs}
 u(t) & =L_t f \pm i^{-\beta} \int_0^t N_{t-t'} (|u(t')|^{p-1} u(t')) \, dt'\\
p_h u_h (t) & = p_h L_{t,h} \Pi_h f_{2h} \pm i^{-\beta} \, \int_0^t p_h N_{t-t',h} \Pi_h R_h (|u_h(t')|^{p-1} u_h(t')) \, dt'.\nonumber
\end{align}

\begin{rk} At this stage, we add the condition: 
\[s+\sigma-\alpha\leq 1.\] 
This is only necessary to guarantee that the regularity with which we work, given by \eqref{eq:norms}, does not exceed the regularity allowed by linear interpolation $p_h u_h$. It would be possible to work on higher regularity than $H^1_x(\R)$ even with piecewise linear functions. However, it is then better to use a quadratic interpolation of $u_h$, which is outside the scope of this paper.
\end{rk}

We now study the difference between $u$ and $K_h^b p_h u_h$ in the norm $\Lambda_T$:
\begin{multline}\label{eq:bootstrap1}
 \Lambda_T (u-K_h^b p_h u_h) \leq \Lambda_T (L_t f -K_h^b p_h L_{t,h} \Pi_h f_{2h}) \\ 
 + \Lambda_T\left(\int_0^t \left[ N_{t-t'} (|u(t')|^{p-1} u(t'))- K_h^b p_h N_{t-t',h} \Pi_h R_h (|u_h(t')|^{p-1} u_h(t'))\right] \, dt'\right)
\end{multline}

We will prove that:

\begin{lem}\label{thm:blabla}
Let $s\geq \frac{1}{2}-\frac{1}{2(p-1)}$ as in \eqref{eq:norms}, and let 
\[\widetilde{s}=\max\{ s+\sigma-\alpha+ , \frac{1}{2}+\}.\]
Then we have 
\begin{multline}\label{eq:finalequation}
 \Lambda_T (u-K_h^b p_h u_h) \leq o_h (1)+ C(T,\norm{f}_{H^{\tilde{s}}}) \cdot \Lambda_T(u-K_h^b p_h u_h)\, \left[1+\Lambda_T(u-K_h^b p_h u_h)^{p-1}\right] .
 \end{multline}
where the constant $C(T,\norm{f}_{H^{\tilde{s}}})\rightarrow 0$ as $T\rightarrow 0$.
\end{lem}

Once we prove this lemma, a standard argument based on the method of continuity shows that \eqref{eq:finalequation} implies the stronger:
\begin{equation}\label{eq:bootstrap}
 \Lambda_T (u- K_h^b p_h u_h)= o_h (1)
\end{equation}
for small enough $T$ and $h$. Finally, \eqref{eq:bootstrap} implies that
\begin{align*}
 \norm{u-p_h u_h}_{L^{\infty}_T H^s_x} & \leq \norm{u-K_h^b p_h u_h}_{L^{\infty}_T H^s_x} + \norm{(1-K_h^b) p_h u_h}_{L^{\infty}_T H^s_x} \\
& \leq \Lambda_T (u- K_h^b p_h u_h) + \norm{ \langle\xi\rangle^{s}\, \widehat{p_h u_h}}_{L^{\infty}_T L^2_{\xi}(|\xi|\gtrsim h^{-b})} \\
& \lesssim o_h (1) + h^{0+} \, \norm{ \langle\xi\rangle^{s+}\, \widehat{p_h u_h}}_{L^{\infty}_T L^2_{\xi}(|\xi|\gtrsim h^{-b})} \\
& \lesssim  o_h(1) + h^{0+} \, \norm{p_h u_h}_{L^{\infty}_T H^{s+}_x}\\
& \lesssim o_h (1) + h^{0+} \, \norm{u_h}_{L^{\infty}_T H^{s+}_h}\\
& \lesssim o_h (1) + h^{0+}\, C\left(T,\norm{f}_{H^{\tilde{s}}}\right)= o_h(1),
\end{align*}
where in the last step we use the LWP theory for $u_h$ (see \Cref{thm:discretelwp}). This proves that $p_h u_h$ converges to the continuum limit $u$, and completes the proof of \Cref{thm:continuum}.

\subsection{Proof of \Cref{thm:blabla}}

We start with the first term in \eqref{eq:bootstrap1}. By \Cref{thm:compareL2}, \Cref{thm:comparesmoothing}, \Cref{thm:comparemaximal} and \Cref{thm:compareint} we have that:
\[ \Lambda_T (L_t f -K_h^b p_h L_{t,h} \Pi_h f_{2h})\lesssim_T h^{\bar{b}} \, \norm{f}_{H^{\tilde{s}}_x}.\]
From now on, the implicit constant when using the symbol $\lesssim$ might depend on $T$ in a polynomial way, so we drop the notation $\lesssim_T$.

We now focus on the nonlinearity in \eqref{eq:bootstrap1}:
\begin{align}
\Lambda_T\left(\int_0^t \left[ N_{t-t'} (|u(t')|^{p-1} u(t')) - K_h^b p_h N_{t-t',h} \Pi_h R_h (|u_h(t')|^{p-1} u_h(t'))\right] \, dt'\right) & \nonumber\\
 &\hspace{-12cm} \leq \Lambda_T\left(\int_0^t \left[ N_{t-t'} (|u(t')|^{p-1} u(t'))- N_{t-t'} K_h^b p_h \Pi_h R_h (|u_h(t')|^{p-1} u_h(t'))\right] \, dt'\right)\nonumber \\
& \hspace{-11.8cm} + \Lambda_T\left(\int_0^t \left[ N_{t-t'} K_h^b p_h \Pi_h R_h (|u_h(t')|^{p-1} u_h(t'))- K_h^b p_h N_{t-t',h} \Pi_h R_h (|u_h(t')|^{p-1} u_h(t'))\right] \, dt'\right)\nonumber\\
& \hspace{-10cm}= I + I\! I.\label{eq:bootstrap2}
\end{align}

First we deal with $I\! I$ since it is easier. \Cref{thm:compareL2}, \Cref{thm:comparesmoothing}, \Cref{thm:comparemaximal} and \Cref{thm:compareint} yield
\begin{align*}
 I\! I & \lesssim h^{\bar{b}} \norm{\Pi_h R_h |u_h|^{p-1} u_h}_{L^2_T H^{s+\sigma-\alpha}_h}\\
& = h^{\bar{b}} \norm{|u_h|^{p-1} u_h }_{L^2_T H^{s+\sigma-\alpha}_h}  \lesssim h^{\bar{b}}\, C(T, \norm{f}_{H^{\tilde{s}}}),
\end{align*}
after using local well-posedness theory for the discrete equation, see \eqref{eq:fixed_point}.

Now we study $I$. Since the continuous operator $N_{t-t'}$ appears on both terms, we may use the local well-posedness theory \eqref{eq:fixed_point} to write:
\begin{align*}
 I \lesssim&\  \norm{ |u|^{p-1} u - K_h^b p_h \Pi_h R_h (|u_h|^{p-1} u_h) }_{L^2_T H^{s+\sigma-\alpha}_x}\\
= &\ \norm{ |u|^{p-1} u - K_h^b p_{2h} R_h (|u_h|^{p-1} u_h) }_{L^2_T H^{s+\sigma-\alpha}_x}\\
 \lesssim &\ \norm{ |u|^{p-1} u-  |K_h^b p_{h} u_h|^{p-1} K_h^b p_{h} u_h }_{L^2_T H^{s+\sigma-\alpha}_x} \\
& +\norm{ |K_h^b p_{h} u_h|^{p-1} K_h^b p_{h} u_h -  |p_{h} u_h|^{p-1} p_{h} u_h }_{L^2_T H^{s+\sigma-\alpha}_x}\\
& + \norm{ |p_h u_h|^{p-1} p_h u_h-  |p_{2h} R_h u_h|^{p-1} p_{2h} R_h u_h }_{L^2_T H^{s+\sigma-\alpha}_x} \\
& + \norm{|p_{2h} R_h u_h|^{p-1} p_{2h} R_h u_h - p_{2h} R_h (|u_h|^{p-1} u_h)}_{L^2_T H^{s+\sigma-\alpha}_x}\\
& + \norm{ p_{2h} R_h (|u_h|^{p-1} u_h) - K_h^b p_{2h} R_h (|u_h|^{p-1} u_h)}_{L^2_T H^{s+\sigma-\alpha}_x}\\
 =&\ I_1 + I_2 + I_3 + I_4 + I_5 .
\end{align*}

We study these terms separately. 

\begin{enumerate}
\item For the first one, we estimate the norm of an expression such as $|u|^{p-1} u - |v|^{p-1} v$ in terms of $u-v$ as in \eqref{eq:ap_bp},
\begin{align*}
 I_1 & =\norm{ |u|^{p-1} u-  |K_h^b p_{h} u_h|^{p-1} K_h^b p_{h} u_h }_{L^2_T H^{s+\sigma-\alpha}_x} \\
 & \lesssim \left( \Lambda_T (u)^{p-1} + \Lambda_T (K_h^b p_{h} u_h)^{p-1} \right) \, \Lambda_T (u - K_h^b p_{h} u_h).
\end{align*}
Now note that 
\[ \Lambda_T (K_h^b p_{h} u_h)\leq \Lambda_T (u-K_h^b p_{h} u_h)+ \Lambda_T (u)\lesssim  \Lambda_T (u-K_h^b p_{h} u_h)+ C(T, \norm{f}_{H^{s}})\]
because the continuous equation is well-posed.
Therefore,
\begin{align*}
 I_1 & \lesssim \left(C(T, \norm{f}_{H^{s}}) + \Lambda_T (u-K_h^b p_{h} u_h)^{p-1} \right) \, \Lambda_T (u - K_h^b p_{h} u_h),
\end{align*}
where $C(T,\norm{f}_{H^{s}})$ is allowed to change from line to line.

\item  Now we study the second term:
\[ I_2=\norm{ |K_h^b p_{h} u_h|^{p-1} K_h^b p_{h} u_h -  |p_{h} u_h|^{p-1} p_{h} u_h }_{L^2_T H^{s+\sigma-\alpha}_x}.\]
We use a rough estimate, together with the discrete Sobolev embedding:
\begin{align*}
 I_2 & \lesssim_T \left( \norm{K_h^b p_h u_h}^{p-1}_{L^{\infty}_{x,T}} + \norm{p_h u_h}^{p-1}_{L^{\infty}_{x,T}} \right) \, \norm{(1- K_h^b) p_{h} u_h}_{L^{\infty}_T H^{s+\sigma-\alpha}_x}\\
& \lesssim \norm{u_h}_{L^{\infty}_{T} H^{\frac{1}{2}+}_h}^{p-1} \ \cdot\, h^{\bar{b}}\, \norm{u_h}_{L^{\infty}_T H^{s+\sigma-\alpha+}_x} \lesssim h^{\bar{b}} \, C(T,\norm{f}_{H^{\tilde{s}}_x}),
\end{align*}
by the discrete local well-posedness theory.

\item For the third term we again use the discrete Sobolev embedding:
\begin{align*}
 I_3 & =\norm{ |p_h u_h|^{p-1} p_h u_h-  |p_{2h} R_h u_h|^{p-1} p_{2h} R_h u_h }_{L^2_T H^{s+\sigma-\alpha}_x}\\
 & \lesssim \left( \norm{p_h u_h}_{L^{\infty}_{x,T}}^{p-1} + \norm{p_{2h} R_h u_h}_{L^{\infty}_{x,T}}^{p-1} \right) \, \norm{p_h u_h - p_{2h} R_h u_h}_{L^{\infty}_T H^{s+\sigma-\alpha}_x}\\
 & \lesssim \left( \norm{u_h}_{L^{\infty}_{T,h}}^{p-1} + \norm{u_h}_{L^{\infty}_{T,h}}^{p-1} \right) \, \norm{p_h u_h - p_{2h} R_h u_h}_{L^{\infty}_T H^{s+\sigma-\alpha}_x}\\
  & \lesssim \norm{u_h}_{L^{\infty}_{T} H^{\frac{1}{2}+}_h}^{p-1} \, \norm{p_h u_h - p_{2h} R_h u_h}_{L^{\infty}_T H^{s+\sigma-\alpha}_x}\\
  & \lesssim h^{\bar{b}}\,\norm{u_h}_{L^{\infty}_{T} H^{\frac{1}{2}+}_h}^{p-1} \, \norm{u_h}_{L^{\infty}_T H^{s+\sigma-\alpha+}_h}.
\end{align*}
The last step, where we gain a small power of $h$, admits a similar proof to \Cref{thm:gainh}.

\item To control the fourth term we use \Cref{thm:honglemma} and the discrete Sobolev embedding,
\begin{align*}
 I_4 & = \norm{|p_{2h} R_h u_h|^{p-1} p_{2h} R_h u_h - p_{2h} R_h (|u_h|^{p-1} u_h)}_{L^2_T H^{s+\sigma-\alpha}_x} \\
 & = \norm{|p_{2h} R_h u_h|^{p-1} p_{2h} R_h u_h - p_{2h} (|R_h u_h|^{p-1} R_h u_h)}_{L^2_T H^{s+\sigma-\alpha}_x}\\
 & \lesssim  h^{\bar{b}}\, \norm{u_h}_{L^{\infty}_T L^{\infty}_h}^{p-1} \, \norm{u_h}_{L^{\infty}_T H^{s+\sigma-\alpha+}_h}\\
  & \lesssim h^{\bar{b}}\, \norm{u_h}_{L^{\infty}_T H^{\frac{1}{2}+}_h}^{p-1} \, \norm{u_h}_{L^{\infty}_T H^{s+\sigma-\alpha+}_h} \lesssim h^{\bar{b}} \, C(T, \norm{f}_{H^{\tilde{s}}_x}).
 \end{align*}

\item  Finally, we study the term:
\begin{align*}
 I_5 & = \norm{ (1-K_h^b) \,  p_{2h} R_h (|u_h|^{p-1} u_h)}_{L^2_T H^{s+\sigma-\alpha}_x} \lesssim h^{\bar{b}}\, \norm{p_{2h} R_h (|u_h|^{p-1} u_h)}_{L^2_T H^{s+\sigma-\alpha+}_x}\\
& \lesssim h^{\bar{b}}\, \norm{ |u_h|^{p-1} u_h}_{L^2_T H^{s+\sigma-\alpha+}_h} \lesssim h^{\bar{b}}\, C(T,\norm{f}_{H^{s}_x}),
\end{align*}
thanks to the local well-posedness theory for the discrete equation \eqref{eq:fixed_point}.
\end{enumerate}

After combining our findings for $I_k$ ($k=1,\ldots, 5)$ and $I\! I$, one obtains \eqref{eq:finalequation}.

\bibliographystyle{hsiam}
\bibliography{references_DFNLS}
\end{document}